\documentclass[12pt,reqno,openbib,runningheads,a4paper]{amsart}%
\usepackage{graphicx}
\usepackage{amssymb}
\usepackage{amsfonts}
\usepackage{amsmath}
\usepackage{graphicx}
\usepackage{lineno}
\usepackage[authoryear]{natbib}
\usepackage[dvipsnames]{xcolor}
\usepackage{epsf}
\usepackage{lscape}
\usepackage[legalpaper,bookmarks=true,colorlinks=true,linkcolor=blue,citecolor=blue]%
{hyperref}%
\providecommand{\U}[1]{\protect\rule{.1in}{.1in}}
\providecommand{\U}[1]{\protect\rule{.1in}{.1in}}
\newtheorem{theorem}{Theorem}[section]
\newtheorem{corollary}{Corollary}[section]

\newtheorem{lemma}{Lemma}[section]

\newtheorem{remark}{Remark}[section]

\makeatletter
\renewcommand{\@biblabel}[1]{}
\makeatother
\begin{document}

\begin{center}
{\Large \textbf{On the estimation of the extreme value index for randomly
right-truncated data and application}}\medskip\medskip

{\large Souad Benchaira, Djamel Meraghni, Abdelhakim Necir}$^{\ast}$\medskip

{\small \textit{Laboratory of Applied Mathematics, Mohamed Khider University,
Biskra, Algeria}}\medskip\medskip%
\[
\]

\end{center}

\noindent\textbf{Abstract}\medskip

\noindent We introduce a consistent estimator of the extreme value index under
random truncation based on a single sample fraction of top observations from
truncated and truncation data. We establish the asymptotic normality of the
proposed estimator by making use of the weighted tail-copula process framework
and we check its finite sample behavior through some simulations. As an
application, we provide asymptotic normality results for an estimator of the
excess-of-loss reinsurance premium.\medskip\medskip

\medskip


\vfill

\vfill

\noindent{\small $^{\text{*}}$Corresponding author:
\texttt{necirabdelhakim@yahoo.fr} \newline\noindent\textit{E-mail
addresses:}\newline\texttt{benchaira.s@hotmail.fr} (S.~Benchaira)\newline%
\texttt{djmeraghni@yahoo.com} (D.~Meraghni)}

\section{\textbf{Introduction\label{sec1}}}

\noindent Let $\left(  \mathbf{X}_{i},\mathbf{Y}_{i}\right)  ,$ $1\leq i\leq
N,$ be $N\geq1$ independent copies from a couple $\left(  \mathbf{X}%
,\mathbf{Y}\right)  $ of independent positive random variables (rv's) defined
over some probability space $\left(  \Omega,\mathcal{A},\mathbf{P}\right)  ,$
with continuous marginal distribution functions (df's) $\mathbf{F}$ and
$\mathbf{G}$ respectively.$\ $ Suppose that $\mathbf{X}$ is right-truncated by
$\mathbf{Y},$ in the sense that $\mathbf{X}_{i}$ is only observed when
$\mathbf{X}_{i}\leq\mathbf{Y}_{i}.$ We assume that both survival functions
$\overline{\mathbf{F}}:=1-\mathbf{F}$\textbf{\ }and $\overline{\mathbf{G}%
}:=1-\mathbf{G}$ are regularly varying at infinity with respective indices
$-1/\gamma_{1}$ and $-1/\gamma_{2}.$ That is, for any $s>0$
\begin{equation}
\lim_{x\rightarrow\infty}\frac{\overline{\mathbf{F}}\left(  sx\right)
}{\overline{\mathbf{F}}\left(  x\right)  }=s^{-1/\gamma_{1}}\text{ and }%
\lim_{y\rightarrow\infty}\frac{\overline{\mathbf{G}}\left(  sy\right)
}{\overline{\mathbf{G}}\left(  y\right)  }=s^{-1/\gamma_{2}}. \label{RV-1}%
\end{equation}
Being characterized by their heavy tails, these distributions play a prominent
role in extreme value theory. They include distributions such as Pareto, Burr,
Fr\'{e}chet, stable and log-gamma, known to be appropriate models for fitting
large insurance claims, log-returns, large fluctuations,
etc...\ \citep[see, e.g.,][]{Res06}. The truncation phenomenon may occur in
many fields, for instance, in insurance it is usual that the insurer's claim
data do not correspond to the underlying losses, because they are truncated
from above. Indeed, when dealing with large claims, the insurance company
stipulates an upper limit to the amounts to be paid out. The excesses over
this fixed threshold are covered by a reinsurance company. This kind of
reinsurance is called excess-loss reinsurance \citep[see, e.g.,][]{RSST-1999}.
Depending on the branches of insurance, the upper limit, which may be random,
is called in different ways: in life insurance, it is called the cedent's
company retention level whereas in non-life insurance, it is called the
deductible. For a recent paper on randomly right-truncated insurance claims,
one refers to \cite{EO-2008}.\medskip

\noindent Let us now denote $\left(  X_{i},Y_{i}\right)  ,$ $i=1,...,n,$ to be
the observed data, as copies of a couple of rv's $\left(  X,Y\right)  $ with
joint df $H,$ corresponding to the truncated sample $\left(  \mathbf{X}%
_{i},\mathbf{Y}_{i}\right)  ,$ $i=1,...,N,$ where $n=n_{N}$ is a sequence of
discrete rv's. By the law of the large numbers, we have $n_{N}/N\overset
{\mathbf{P}}{\rightarrow}p:=\mathbf{P}\left(  \mathbf{X}\leq\mathbf{Y}\right)
,$ as $N\rightarrow\infty.$ For convenience, we use, throughout the paper, the
notation $n\rightarrow\infty$ to say that $n\overset{\mathbf{P}}{\rightarrow
}\infty.$ For $x,y\geq0,$ we have%
\[%
\begin{array}
[c]{ll}%
H\left(  x,y\right)  & :=\mathbf{P}\left(  X\leq x,Y\leq y\right)  \medskip\\
& =\mathbf{P}\left(  \mathbf{X}\leq x,\mathbf{Y}\leq y\mid\mathbf{X}%
\leq\mathbf{Y}\right)  =p^{-1}%
{\displaystyle\int_{0}^{x}}
\mathbf{F}\left(  \min\left(  y,z\right)  \right)  d\mathbf{G}\left(
z\right)  .
\end{array}
\]
Note that, conditionally on $n,$ the observed data are still independent. The
marginal distributions of the observed $X^{\prime}s$ and $Y^{\prime}s,$
respectively denoted by $F$ and $G,$ are equal to%
\[
F\left(  x\right)  =p^{-1}%
{\displaystyle\int_{0}^{x}}
\overline{\mathbf{G}}\left(  z\right)  d\mathbf{F}\left(  z\right)  \text{ and
}G\left(  y\right)  =p^{-1}\int_{0}^{y}\mathbf{F}\left(  z\right)
d\mathbf{G}\left(  z\right)  ,
\]
it follows that the corresponding tails%
\[
\overline{F}\left(  x\right)  =-p^{-1}%
{\displaystyle\int_{x}^{\infty}}
\overline{\mathbf{G}}\left(  z\right)  d\overline{\mathbf{F}}\left(  z\right)
\text{ and }\overline{G}\left(  y\right)  =-p^{-1}%
{\displaystyle\int_{y}^{\infty}}
\mathbf{F}\left(  z\right)  d\overline{\mathbf{G}}\left(  z\right)  .
\]
It is clear that the asymptotic behavior of $\overline{F}$ simultaneously
depends on $\overline{\mathbf{G}}$ and $\overline{\mathbf{F}}$ while that of
$\overline{G}$ only relies on $\overline{\mathbf{G}}\mathbf{.}$ Making use of
Potter's bound inequalities (see Lemma \ref{Lem3}), for the regularly varying
functions $\overline{\mathbf{F}}$ and $\overline{\mathbf{G}},$ we may readily
show that both $\overline{G}$ and $\overline{F}$ are regularly varying at
infinity as well, with respective indices $\gamma_{2}$ and $\gamma:=\gamma
_{1}\gamma_{2}/\left(  \gamma_{1}+\gamma_{2}\right)  .$ That is, we have, for
any $s>0,$%
\begin{equation}
\lim_{x\rightarrow\infty}\frac{\overline{F}\left(  sx\right)  }{\overline
{F}\left(  x\right)  }=s^{-1/\gamma}\text{ and }\lim_{y\rightarrow\infty}%
\frac{\overline{G}\left(  sy\right)  }{\overline{G}\left(  y\right)
}=s^{-1/\gamma_{2}}. \label{RV-2}%
\end{equation}
Recently \cite{GS2014} addressed the estimation of the extreme value index
$\gamma_{1}$ under random truncation. They used the definition of $\gamma$ to
derive the following consistent estimator:%
\[
\widehat{\gamma}_{1}\left(  k,k^{\prime}\right)  :=\frac{\widehat{\gamma
}\left(  k\right)  \widehat{\gamma}_{2}\left(  k^{\prime}\right)  }%
{\widehat{\gamma}_{2}\left(  k^{\prime}\right)  -\widehat{\gamma}\left(
k\right)  },
\]
where%
\begin{equation}
\widehat{\gamma}\left(  k\right)  :=\frac{1}{k}\sum_{i=1}^{k}\log
\frac{X_{n-i+1:n}}{X_{n-k:n}}\text{ and }\widehat{\gamma}_{2}\left(
k^{\prime}\right)  :=\frac{1}{k^{\prime}}\sum_{i=1}^{k^{\prime}}\log
\frac{Y_{n-i+1:n}}{Y_{n-k^{\prime}:n}}, \label{Hills}%
\end{equation}
are the well-known Hill estimators of $\gamma$ and $\gamma_{2},$ with
$X_{1:n}\leq...\leq X_{n:n}$ and $Y_{1:n}\leq...\leq Y_{n:n}$ being the order
statistics pertaining to the samples $\left(  X_{1},...,X_{n}\right)  $ and
$\left(  Y_{1},...,Y_{n}\right)  $ respectively. The two sequences $k=k_{n}$
and $k^{\prime}=k_{n}^{\prime}$ of integer rv's, which satisfy%
\[
1<k,k^{\prime}<n,\text{ }\min\left(  k,k^{\prime}\right)  \rightarrow
\infty\text{ and }\max\left(  k/n,k^{\prime}/n\right)  \rightarrow0\text{ as
}n\rightarrow\infty,
\]
respectively represent the numbers of top observations from truncated and
truncation data. By considering the two situations $k/k^{\prime}\rightarrow0$
and $k^{\prime}/k\rightarrow0$ as $n\rightarrow\infty,$ the authors
established the asymptotic normality of $\widehat{\gamma}_{1}\left(
k,k^{\prime}\right)  ,$ but when $k/k^{\prime}\rightarrow1,$ they only showed
that%
\[
\sqrt{\min\left(  k,k^{\prime}\right)  }\left(  \widehat{\gamma}_{1}\left(
k,k^{\prime}\right)  -\gamma_{1}\right)  =O_{\mathbf{p}}\left(  1\right)
,\text{ as }n\rightarrow\infty,
\]
which is not enough to construct confidence intervals for $\gamma_{1}.$ It is
obvious that an accurate computation of the estimate $\widehat{\gamma}%
_{1}\left(  k,k^{\prime}\right)  $ requires good choices of both $k$ and
$k^{\prime}.$ However from a practical point of view, it is rather unusual in
extreme value analysis to handle two distinct sample fractions simultaneously,
which is mentioned by \cite{GS2014} in their conclusion as well. For this
reason, we consider, in the present work, the situation when $k=k^{\prime}$
rather than $k/k^{\prime}\rightarrow1.$ Thus, we obtain an estimator%
\begin{equation}
\widehat{\gamma}_{1}:=\widehat{\gamma}_{1}\left(  k\right)  =k^{-1}\frac{%
{\displaystyle\sum\limits_{i=1}^{k}}
\log\dfrac{X_{n-i+1:n}}{X_{n-k:n}}%
{\displaystyle\sum\limits_{i=1}^{k}}
\log\dfrac{Y_{n-i+1:n}}{Y_{n-k:n}}}{%
{\displaystyle\sum\limits_{i=1}^{k}}
\log\dfrac{X_{n-k:n}Y_{n-i+1:n}}{Y_{n-k:n}X_{n-i+1:n}}}, \label{gamma(k)}%
\end{equation}
of simpler form, expressed in terms of a single sample fraction $k$ of
truncated and truncation observations. The number of extreme values used to
compute the optimal estimate value $\widehat{\gamma}_{1}$ may be obtained by
applying one of the various heuristic methods available in the literature such
that, for instance, the algorithm of page 137 in \cite{ReTo7}, which will be
applied in Section \ref{sec3}.\medskip

\noindent The task of establishing the asymptotic normality of $\widehat
{\gamma}_{1}$ is a bit delicate as one has to take into account the dependence
structure of $X$ and $Y.$ The authors of \cite{GS2014} avoided this issue by
putting conditions on the sample fractions $k$ and $k^{\prime}.$ In our case
we require that the joint df $H$ have a stable tail dependence function $\ell$
(see \citeauthor{Huang-1992}, \citeyear{Huang-1992} and \citeauthor{DrHu98},
\citeyear[]{DrHu98}), in the sense that the following limit exists:%
\begin{equation}
\lim_{t\downarrow0}t^{-1}\mathbf{P}\left(  \overline{F}\left(  X\right)  \leq
tx\text{ or }\overline{G}\left(  Y\right)  \leq ty\right)  =:\ell\left(
x,y\right)  , \label{L}%
\end{equation}
for all $x,y\geq0$ such that $\max\left(  x,y\right)  >0.$ Note that the
corresponding tail copula function is defined by
\begin{equation}
\lim_{t\downarrow0}t^{-1}\mathbf{P}\left(  \overline{F}\left(  X\right)  \leq
tx,\text{ }\overline{G}\left(  Y\right)  \leq ty\right)  =:R\left(
x,y\right)  , \label{R}%
\end{equation}
which equals $x+y-\ell\left(  x,y\right)  .$ In on other words, we assume that
$H$ belongs to the domain of attraction of a bivariate extreme value
distribution. This may be split into two sets of conditions, namely conditions
for the convergence of the marginal distributions $\left(  \ref{RV-2}\right)
$ and others for the convergence of the dependence structure $\left(
\ref{L}\right)  .$ For details on this topic, see for instance Section 6.1.2
of \cite{deHF06} and the papers of \cite{Huang-1992}, \cite{ScSt2006},
\cite{EHL-2006}, \cite{deHNePe2008} and \cite{Peng2010}.\medskip

\noindent The rest of the paper is organized as follows.\ In Section
\ref{sec2}, we give our main result which consists in a Gaussian approximation
to $\widehat{\gamma}_{1}$ only by assuming the second-order conditions of
regular variation and the stability of the tail dependence function. A
simulation study is carried out, in Section \ref{sec3}, to illustrate the
performance of $\widehat{\gamma}_{1}.$ Section \ref{sec4} is devoted to an
application, as we derive the asymptotic normality of an excess-of-loss
reinsurance premium estimator. Finally, the proofs are postponed to Section
\ref{sec5} whereas some results that are instrumental to our needs are
gathered in the Appendix.

\section{\textbf{Main results\label{sec2}}}

\noindent Weak approximations of extreme value theory based statistics are
achieved in the second-order framework \citep[see][]{deHS96}. Thus, it seems
quite natural to suppose that both df's $F$ and $G$ satisfy the well-known
second-order condition of regular variation. That is, we assume that for any
$x>0$%
\begin{equation}%
\begin{array}
[c]{c}%
\underset{z\rightarrow\infty}{\lim}\dfrac{1}{A\left(  z\right)  }\left(
\dfrac{U\left(  zx\right)  }{U\left(  z\right)  }-x^{\gamma}\right)
=x^{\gamma}\dfrac{x^{\tau}-1}{\tau},\medskip\\
\underset{z\rightarrow\infty}{\lim}\dfrac{1}{A_{2}\left(  z\right)  }\left(
\dfrac{U_{2}\left(  zx\right)  }{U_{2}\left(  z\right)  }-x^{\gamma_{2}%
}\right)  =x^{\gamma_{2}}\dfrac{x^{\tau_{2}}-1}{\tau_{2}},
\end{array}
\label{second-order}%
\end{equation}
where $U:=\left(  1/\overline{F})\right)  ^{\leftarrow},$\ $U_{2}:=\left(
1/\overline{G})\right)  ^{\leftarrow}$ (with $E^{\leftarrow}\left(  u\right)
:=\inf\left\{  v:E\left(  v\right)  \geq u\right\}  ,$ for $0<u<1,$ denoting
the quantile function pertaining to a function $E),$ $\left\vert A\right\vert
$ and $\left\vert A_{2}\right\vert $ are some regularly varying functions with
negative indices (second-order parameters) $\tau$ and $\tau_{2}$ respectively.

\begin{theorem}
\textbf{\label{Theorem1}}Assume that the second-order regular variation
condition $(\ref{second-order})$ and $(\ref{L})$ hold.$\ $Let $k:=k_{n}$ be a
sequence of integers such that $k\rightarrow\infty,$ $k/n\rightarrow0,$
$\sqrt{k}A\left(  n/k\right)  =O\left(  1\right)  =\sqrt{k}A_{2}\left(
n/k\right)  .\ $Then, there exist two standard Wiener processes $\left\{
W_{i}\left(  t\right)  ,\text{ }t\geq0\right\}  ,$ $i=1,2,$ defined on the
probability space $\left(  \Omega,\mathcal{A},\mathbf{P}\right)  $ with
covariance function $R\left(  \cdot,\cdot\right)  ,$ such that%
\[
\sqrt{k}\left(  \frac{X_{n-k:n}}{U\left(  n/k\right)  }-1\right)  -\gamma
W_{1}\left(  1\right)  =o_{\mathbf{p}}\left(  1\right)  =\sqrt{k}\left(
\frac{Y_{n-k:n}}{U_{2}\left(  n/k\right)  }-1\right)  -\gamma_{2}W_{2}\left(
1\right)  ,
\]
and%
\begin{align*}
&  \sqrt{k}\left(  \widehat{\gamma}_{1}-\gamma_{1}\right)  -\mu\left(
k\right) \\
&  =\int_{0}^{1}t^{-1}\left(  cW_{1}\left(  t\right)  -c_{2}W_{2}\left(
t\right)  \right)  dt-cW_{1}\left(  1\right)  +c_{2}W_{2}\left(  1\right)
+o_{\mathbf{p}}\left(  1\right)  ,
\end{align*}
where $c:=\gamma_{1}^{2}/\gamma,$ $c_{2}:=\gamma_{1}^{2}/\gamma_{2}$ and%
\[
\mu\left(  k\right)  :=\dfrac{c\sqrt{k}A\left(  n/k\right)  }{\gamma\left(
1-\tau\right)  }+\dfrac{c_{2}\sqrt{k}A_{2}\left(  n/k\right)  }{\gamma
_{2}\left(  1-\tau_{2}\right)  }.
\]

\end{theorem}

\begin{corollary}
\textbf{\label{cor1}}Under the assumptions of Theorem $\ref{Theorem1},$ we
have
\[
\sqrt{k}\left(  \widehat{\gamma}_{1}-\gamma_{1}\right)  \overset{\mathcal{D}%
}{\rightarrow}\mathcal{N}\left(  \mu,\sigma^{2}\right)  ,\text{ as
}n\rightarrow\infty,
\]
provided that $\sqrt{k}A\left(  n/k\right)  \rightarrow\lambda$ and $\sqrt
{k}A_{2}\left(  n/k\right)  \rightarrow\lambda_{2},$ where%
\[
\mu:=\frac{c\lambda}{\gamma\left(  1-\tau\right)  }+\frac{c_{2}\lambda_{2}%
}{\gamma_{2}\left(  1-\tau_{2}\right)  }\text{ and }\sigma^{2}:=2c^{2}%
+2c_{2}^{2}-2cc_{2}\delta,
\]
with%
\[
\delta=\delta\left(  R\right)  :=\int_{0}^{1}\int_{0}^{1}\frac{R\left(
s,t\right)  }{st}dsdt-\int_{0}^{1}\left(  R\left(  s,1\right)  -R\left(
1,s\right)  \right)  ds+R\left(  1,1\right)  .
\]

\end{corollary}

\begin{remark}
Note that $\sigma^{2}$ is finite. Indeed, the fact that $\ell\left(
x,y\right)  $ is a tail copula function, implies that $\max\left(  x,y\right)
\leq\ell\left(  x,y\right)  \leq x+y$ \citep[see, e.g.,][]{GUS2010} and since
$R\left(  x,y\right)  =x+y-\ell\left(  x,y\right)  ,$ then $0\leq R\left(
x,y\right)  \leq\min\left(  x,y\right)  .$ It follows that%
\[
\int_{0}^{1}\int_{0}^{1}\frac{R\left(  s,t\right)  }{st}dsdt\leq\int_{0}%
^{1}\int_{0}^{1}\frac{\min\left(  s,t\right)  }{st}dsdt=2,
\]%
\[
\int_{0}^{1}R\left(  s,1\right)  ds\leq\frac{1}{2},\text{ }\int_{0}%
^{1}R\left(  1,s\right)  ds\leq\frac{1}{2}\text{ and }R\left(  1,1\right)
\leq1.
\]
Therefore $\left\vert \delta\right\vert \leq4,$ which yields that $\sigma
^{2}<\infty.$
\end{remark}

\noindent The following corollary directly leads to a practical construction
of confidence intervals for the tail index $\gamma_{1}.$

\begin{corollary}
\textbf{\label{cor2}}Under the assumptions of Corollary $\ref{cor1},$ we have%
\[
\frac{\sqrt{k}\left(  \widehat{\gamma}_{1}-\gamma_{1}\right)  -\widehat{\mu}%
}{\widehat{\sigma}}\overset{\mathcal{D}}{\rightarrow}\mathcal{N}\left(
0,1\right)  ,\text{ as }n\rightarrow\infty,
\]
where $\widehat{\mu}=\dfrac{\widehat{c}\widehat{\lambda}}{\widehat{\gamma
}\left(  1-\widehat{\tau}\right)  }+\dfrac{\widehat{c}_{2}\widehat{\lambda
}_{2}}{\widehat{\gamma}_{2}\left(  1-\widehat{\tau}_{2}\right)  }$ and
$\widehat{\sigma}^{2}:=2\widehat{c}^{2}+2\widehat{c}_{2}^{2}-2\widehat
{c}\widehat{c}_{2}\widehat{\delta},$ with
\[
\widehat{c}:=\widehat{\gamma}_{1}^{2}/\widehat{\gamma},\text{ }\widehat{c}%
_{2}:=\widehat{\gamma}_{1}^{2}/\widehat{\gamma}_{2},\text{ }\widehat{\delta
}:=\delta\left(  \widehat{R}\right)  ,
\]%
\[
\widehat{\lambda}:=\sqrt{k}\widehat{\tau}\frac{X_{n-2k:n}-2^{-\widehat{\gamma
}}X_{n-k:n}}{2^{-\widehat{\gamma}}\left(  2^{-\widehat{\tau}}-1\right)
X_{n-k:n}}\text{ and }\widehat{\lambda}_{2}:=\sqrt{k}\widehat{\tau}_{2}%
\frac{Y_{n-2k:n}-2^{-\widehat{\gamma}_{2}}Y_{n-k:n}}{2^{-\widehat{\gamma}_{2}%
}\left(  2^{-\widehat{\tau}_{2}}-1\right)  Y_{n-k:n}}.
\]
Here $\widehat{\gamma}$ and $\widehat{\gamma}_{2}$ are the respective Hill
estimators of $\gamma$ and $\gamma_{2}$ defined in $\left(  \ref{Hills}%
\right)  $ with $k^{\prime}=k,$ $\widehat{\tau}$ (resp. $\widehat{\tau}_{2})$
is one of the estimators of $\tau$ (resp. $\tau_{2})$
\citep[see, e.g.,][]{GP07} and $\widehat{R}$ is a nonparametric estimator of
$R$ given in \cite{Peng2010} by $\widehat{R}\left(  s,t\right)  :=k^{-1}%
\sum_{i=1}^{n}\mathbf{1}\left(  X_{i}\geq X_{n-\left[  ks\right]  :n},\text{
}Y_{i}\geq Y_{n-\left[  kt\right]  :n}\right)  ,$ with $\left[  x\right]  $
standing for the integer part of the real number $x$ and $\mathbf{1}\left(
\cdot\right)  $ for the indicator function.
\end{corollary}

\section{Simulation study\textbf{\label{sec3}}}

\noindent We carry out a simulation study to illustrate the performance of our
estimator, through two sets of truncated and truncation data, both drawn from
Burr's model. We have%
\[
\mathbf{F}\left(  x\right)  =1-\left(  1+x^{1/\delta}\right)  ^{-\delta
/\gamma_{1}}\text{ and }\mathbf{G}\left(  y\right)  =1-\left(  1+y^{1/\delta
}\right)  ^{-\delta/\gamma_{2}},\text{ }x,y>0,
\]
with $\delta>0$ and $0<\gamma_{1}<\gamma_{2}.$ The second-order parameters of
$\left(  \ref{second-order}\right)  $ are $\tau=-2\gamma/\delta$ and $\tau
_{2}=-\gamma_{2}/\delta.$ The truncation probability is equal to $1-p$ with
$p=\gamma_{2}/\left(  \gamma_{1}+\gamma_{2}\right)  .$ We fix $p=0.7,$ $0.8,$
$0.9$ and $\gamma_{1}=0.6,$ $0.8.\ $The corresponding $\gamma_{2}-$values are
obtained by solving the latter equation. We vary the common size $N$ of both
samples and for each size, we generate $200$ independent replicates. Our
overall results are then taken as the empirical means of the values obtained
in the $200$ repetitions. To determine the optimal number of upper order
statistics used in the computation of $\widehat{\gamma}_{1},$ we apply the
algorithm of page 137 in \cite{ReTo7}.$\medskip$

\noindent This study consists in two parts: point estimation and
$95\%-$confidence interval construction. In the first part, we evaluate the
bias and the root of the mean squared error (rmse) of $\widehat{\gamma}_{1}$
while in the second, we investigate the accuracy of the confidence intervals
of the tail index $\gamma_{1},$ by computing their lengths and coverage
probabilities (denoted by `covpr'). The results of the first part are
summarized in Table \ref{EP}, whereas those of the second are given in Table
\ref{IC}, where `lcb' and `ucb' respectively stand for the lower and upper
confidence bounds. To compute confidence bounds for $\gamma_{1},$ with level
$\left(  1-\zeta\right)  \times100\%$ where $0<\zeta<1,$ from two realizations
$\left(  x_{1},...,x_{n}\right)  $ and $\left(  y_{1},...,y_{n}\right)  $ of
$\left(  X_{1},...,X_{n}\right)  $ and $\left(  Y_{1},...,Y_{n}\right)  $
respectively, we use Corollary \ref{cor2} and proceed as follows$.$

\begin{itemize}
\item Select the optimal sample fraction of top statistics that we denote by
$k^{\ast}.$

\item Compute the corresponding $\gamma_{1}^{\ast}=\widehat{\gamma}_{1}\left(
k^{\ast}\right)  ,$ $\gamma_{2}^{\ast}=\widehat{\gamma}_{2}\left(  k^{\ast
}\right)  ,$ $\gamma^{\ast}=\widehat{\gamma}\left(  k^{\ast}\right)  ,$
$c^{\ast}=\widehat{c}\left(  k^{\ast}\right)  $ and $c_{2}^{\ast}=\widehat
{c}_{2}\left(  k^{\ast}\right)  .$

\item Calculate $\tau^{\ast}=\widehat{\tau}\left(  k^{\ast}\right)  $ and
$\tau_{2}^{\ast}=\widehat{\tau}_{2}\left(  k^{\ast}\right)  $ via one of the
available numerical procedures (see, e.g., \citeauthor{GP07}, \citeyear{GP07})
and then get $\lambda^{\ast}=\widehat{\lambda}\left(  k^{\ast}\right)  ,$
$\lambda_{2}^{\ast}=\widehat{\lambda}_{2}\left(  k^{\ast}\right)  .$

\item Evaluate $\delta^{\ast}=\widehat{\delta}\left(  k^{\ast}\right)  $ by
means of Monte Carlo integration.

\item Compute $\mu^{\ast}=\widehat{\mu}\left(  k^{\ast}\right)  $ and
$\sigma^{\ast}=\sigma\left(  k^{\ast}\right)  .$
\end{itemize}

\noindent At last, the $\left(  1-\zeta\right)  \times100\%-$confidence bounds
for the extreme value index $\gamma_{1}$ are%
\[
\gamma_{1}^{\ast}+\frac{1}{\sqrt{k^{\ast}}}\left(  \mu^{\ast}\pm\sigma^{\ast
}z_{\zeta/2}\right)  ,
\]
where $z_{\zeta/2}$ is the $(1-\zeta/2)-$quantile of the standard normal
rv.$\medskip$

\noindent On the light of the results of both tables, we see that truncation
is the factor that affects most the estimation process of the tail index. As
we would have expected, the smaller the truncation percentage is, the better
and more accurate the estimation is, for both index values and each sample
size. The reason why we don't consider small samples (we start with a size of
$500)$ is that, in extreme-value theory based inference, large samples are
needed in order for the results to be significant. This motivation becomes
more obvious when, in addition, there is truncation.%

\begin{table}[tbp] \centering
\begin{tabular}
[c]{ccccccccccc}\hline
\multicolumn{11}{c}{$p=0.70$}\\\hline
$N$ & \multicolumn{1}{|c}{$n$} & $k$ & $\widehat{\gamma}_{1}$ & bias & rmse &
\multicolumn{1}{||c}{$n$} & $k$ & $\widehat{\gamma}_{1}$ & bias & rmse\\\hline
\multicolumn{1}{r}{$500$} & \multicolumn{1}{|r}{$350$} &
\multicolumn{1}{r}{$15$} & \multicolumn{1}{r}{$0.515$} &
\multicolumn{1}{r}{$-0.084$} & \multicolumn{1}{r|}{$0.299$} &
\multicolumn{1}{||r}{$349$} & \multicolumn{1}{r}{$15$} &
\multicolumn{1}{r}{$0.673$} & \multicolumn{1}{r}{$-0.127$} &
\multicolumn{1}{r}{$0.356$}\\
\multicolumn{1}{r}{$1000$} & \multicolumn{1}{|r}{$701$} &
\multicolumn{1}{r}{$35$} & \multicolumn{1}{r}{$0.555$} &
\multicolumn{1}{r}{$-0.044$} & \multicolumn{1}{r}{$0.264$} &
\multicolumn{1}{||r}{$699$} & \multicolumn{1}{r}{$32$} &
\multicolumn{1}{r}{$0.704$} & \multicolumn{1}{r}{$-0.095$} &
\multicolumn{1}{r}{$0.307$}\\
\multicolumn{1}{r}{$1500$} & \multicolumn{1}{|r}{$1049$} &
\multicolumn{1}{r}{$50$} & \multicolumn{1}{r}{$0.554$} &
\multicolumn{1}{r}{$-0.046$} & \multicolumn{1}{r|}{$0.212$} &
\multicolumn{1}{||r}{$1049$} & \multicolumn{1}{r}{$51$} &
\multicolumn{1}{r}{$0.751$} & \multicolumn{1}{r}{$-0.049$} &
\multicolumn{1}{r}{$0.259$}\\\hline
\multicolumn{11}{c}{$p=0.80$}\\\hline
\multicolumn{1}{r}{$500$} & \multicolumn{1}{|r}{$400$} &
\multicolumn{1}{r}{$18$} & \multicolumn{1}{r}{$0.521$} &
\multicolumn{1}{r}{$-0.079$} & \multicolumn{1}{r}{$0.233$} &
\multicolumn{1}{||r}{$400$} & \multicolumn{1}{r}{$18$} &
\multicolumn{1}{r}{$0.723$} & \multicolumn{1}{r}{$-0.077$} &
\multicolumn{1}{r}{$0.351$}\\
\multicolumn{1}{r}{$1000$} & \multicolumn{1}{|r}{$801$} &
\multicolumn{1}{r}{$43$} & \multicolumn{1}{r}{$0.566$} &
\multicolumn{1}{r}{$-0.034$} & \multicolumn{1}{r}{$0.181$} &
\multicolumn{1}{||r}{$799$} & \multicolumn{1}{r}{$40$} &
\multicolumn{1}{r}{$0.713$} & \multicolumn{1}{r}{$-0.087$} &
\multicolumn{1}{r}{$0.273$}\\
\multicolumn{1}{r}{$1500$} & \multicolumn{1}{|r}{$1198$} &
\multicolumn{1}{r}{$64$} & \multicolumn{1}{r}{$0.566$} &
\multicolumn{1}{r}{$-0.033$} & \multicolumn{1}{r}{$0.145$} &
\multicolumn{1}{||r}{$1200$} & \multicolumn{1}{r}{$64$} &
\multicolumn{1}{r}{$0.752$} & \multicolumn{1}{r}{$-0.048$} &
\multicolumn{1}{r}{$0.203$}\\\hline
\multicolumn{11}{c}{$p=0.90$}\\\hline
\multicolumn{1}{r}{$500$} & \multicolumn{1}{|r}{$450$} &
\multicolumn{1}{r}{$22$} & \multicolumn{1}{r}{$0.547$} &
\multicolumn{1}{r}{$-0.053$} & \multicolumn{1}{r}{$0.186$} &
\multicolumn{1}{||r}{$449$} & \multicolumn{1}{r}{$20$} &
\multicolumn{1}{r}{$0.702$} & \multicolumn{1}{r}{$-0.098$} &
\multicolumn{1}{r}{$0.295$}\\
\multicolumn{1}{r}{$1000$} & \multicolumn{1}{|r}{$900$} &
\multicolumn{1}{r}{$45$} & \multicolumn{1}{r}{$0.558$} &
\multicolumn{1}{r}{$-0.042$} & \multicolumn{1}{r}{$0.148$} &
\multicolumn{1}{||r}{$900$} & \multicolumn{1}{r}{$49$} &
\multicolumn{1}{r}{$0.747$} & \multicolumn{1}{r}{$-0.053$} &
\multicolumn{1}{r}{$0.189$}\\
\multicolumn{1}{r}{$1500$} & \multicolumn{1}{|r}{$1349$} &
\multicolumn{1}{r}{$76$} & \multicolumn{1}{r}{$0.577$} &
\multicolumn{1}{r}{$-0.023$} & \multicolumn{1}{r}{$0.118$} &
\multicolumn{1}{||r}{$1348$} & \multicolumn{1}{r}{$77$} &
\multicolumn{1}{r}{$0.755$} & \multicolumn{1}{r}{$-0.045$} &
\multicolumn{1}{r}{$0.151$}\\\hline
\multicolumn{1}{l}{} & \multicolumn{1}{l}{} & \multicolumn{1}{l}{} &
\multicolumn{1}{l}{} & \multicolumn{1}{l}{} & \multicolumn{1}{l}{} &
\multicolumn{1}{l}{} & \multicolumn{1}{l}{} & \multicolumn{1}{l}{} &
\multicolumn{1}{l}{} & \multicolumn{1}{l}{}%
\end{tabular}
\caption{Point estimation of the tail index based on 200 samples from randomly right-truncated
Burr population with shape parameter 0.6 (left panel) and 0.8 (right panel)}\label{EP}%
\end{table}%
%

\begin{table}[tbp] \centering
\begin{tabular}
[c]{cccccll}\hline
\multicolumn{7}{c}{$p=0.70$}\\\hline
$N$ & \multicolumn{1}{|c}{lcb$-$ucb} & covpr & length &
\multicolumn{1}{||c}{lcb$-$ucb} & covpr & length\\\hline
\multicolumn{1}{r}{$500$} & \multicolumn{1}{r}{$0.226-0.993$} &
\multicolumn{1}{r}{$0.94$} & \multicolumn{1}{r}{$0.767$} &
\multicolumn{1}{||r}{$0.294-1.199$} & \multicolumn{1}{r}{$0.92$} &
\multicolumn{1}{r}{$0.905$}\\
\multicolumn{1}{r}{$1000$} & \multicolumn{1}{r}{$0.319-0.851$} &
\multicolumn{1}{r}{$0.94$} & \multicolumn{1}{r}{$0.532$} &
\multicolumn{1}{||r}{$0.428-1.099$} & \multicolumn{1}{r}{$0.91$} &
\multicolumn{1}{r}{$0.671$}\\
\multicolumn{1}{r}{$1500$} & \multicolumn{1}{r}{$0.396-0.801$} &
\multicolumn{1}{r}{$0.90$} & \multicolumn{1}{r}{$0.405$} &
\multicolumn{1}{||r}{$0.516-1.043$} & \multicolumn{1}{r}{$0.89$} &
\multicolumn{1}{r}{$0.527$}\\\hline
\multicolumn{7}{c}{$p=0.80$}\\\hline
\multicolumn{1}{r}{$500$} & \multicolumn{1}{|r}{$0.242-0.880$} &
\multicolumn{1}{r}{$0.93$} & \multicolumn{1}{r}{$0.638$} &
\multicolumn{1}{||r}{$0.313-1.158$} & \multicolumn{1}{r}{$0.93$} &
\multicolumn{1}{r}{$0.845$}\\
\multicolumn{1}{r}{$1000$} & \multicolumn{1}{|r}{$0.376-0.790$} &
\multicolumn{1}{r}{$0.92$} & \multicolumn{1}{r}{$0.414$} &
\multicolumn{1}{||r}{$0.497-1.041$} & \multicolumn{1}{r}{$0.92$} &
\multicolumn{1}{r}{$0.544$}\\
\multicolumn{1}{r}{$1500$} & \multicolumn{1}{|r}{$0.436-0.759$} &
\multicolumn{1}{r}{$0.91$} & \multicolumn{1}{r}{$0.323$} &
\multicolumn{1}{||r}{$0.576-1.003$} & \multicolumn{1}{r}{$0.91$} &
\multicolumn{1}{r}{$0.427$}\\\hline
\multicolumn{7}{c}{$p=0.90$}\\\hline
\multicolumn{1}{r}{$500$} & \multicolumn{1}{|r}{$0.317-0.820$} &
\multicolumn{1}{r}{$0.90$} & \multicolumn{1}{r}{$0.503$} &
\multicolumn{1}{||r}{$0.438-1.085$} & \multicolumn{1}{r}{$0.92$} &
\multicolumn{1}{r}{$0.647$}\\
\multicolumn{1}{r}{$1000$} & \multicolumn{1}{|r}{$0.408-0.750$} &
\multicolumn{1}{r}{$0.91$} & \multicolumn{1}{r}{$0.342$} &
\multicolumn{1}{||r}{$0.540-0.998$} & \multicolumn{1}{r}{$0.90$} &
\multicolumn{1}{r}{$0.458$}\\
\multicolumn{1}{r}{$1500$} & \multicolumn{1}{|r}{$0.445-0.727$} &
\multicolumn{1}{r}{$0.90$} & \multicolumn{1}{r}{$0.282$} &
\multicolumn{1}{||r}{$0.591-0.965$} & \multicolumn{1}{r}{$0.90$} &
\multicolumn{1}{r}{$0.374$}\\\hline
\multicolumn{1}{l}{} & \multicolumn{1}{l}{} & \multicolumn{1}{l}{} &
\multicolumn{1}{l}{} & \multicolumn{1}{l}{} &  &
\end{tabular}
\caption{Accuracy of $95\%$-confidence intervals for the tail index based on 200
samples from randomly right-truncated Burr population with shape parameter 0.6 (left panel) and 0.8 (right panel)}\label{IC}%
\end{table}%

\section{\textbf{Application: excess-of-loss reinsurance premium
estimation\label{sec4}}}

\noindent As an application of Theorem \ref{Theorem1}, we derive the
asymptotic normality of an estimator of the excess-of-loss reinsurance premium
obtained with truncated data. Our choice is motivated mainly by two reasons.
The first one is that reinsurance is a very important field of application of
extreme value theory and the second is that data sets with truncated extreme
observations may very likely be encountered in insurance. The aim of
reinsurance, where emphasis lies on modelling extreme events, is to protect an
insurance company, called ceding company, against losses caused by excessively
large claims and/or a surprisingly high number of moderate claims. Nice
discussions on the use of extreme value theory in the actuarial world
(especially in the reinsurance industry) can be found, for instance, in
\cite{EKM-1997}, a major textbook on the subject, and \cite{BeGeS04}.\medskip

\noindent Let $\mathbf{X}_{1},...,\mathbf{X}_{n}$ $\left(  n\geq1\right)  $ be
$n$ individual claim amounts of an insured heavy-tailed loss $\mathbf{X}$ with
finite mean. A Pareto-like distribution, with tail index greater than or equal
to $1,$ does not have finite mean. Hence, assuming that $\mathbf{E}\left[
\mathbf{X}\right]  $ exists necessarily implies that $\gamma_{1}<1.$ In the
excess-of-loss reinsurance treaty, the ceding company covers claims that do
not exceed a (high) number $u\geq0,$ called retention level, while the
reinsurer pays the part $(\mathbf{X}_{i}-u)_{+}:=\max\left(  0,\mathbf{X}%
_{i}-u\right)  $ of each claim exceeding $u.$ The net premium for the layer
from $u$ to infinity is defined as follows:%
\[
\Pi=\Pi(u):=\mathbf{E}\left[  (\mathbf{X}-u)_{+}\right]  =\int_{u}^{\infty
}\overline{\mathbf{F}}\left(  x\right)  dx,
\]
which may be rewritten into $\Pi=u\overline{\mathbf{F}}\left(  u\right)
\int_{1}^{\infty}\overline{\mathbf{F}}\left(  ux\right)  /\overline
{\mathbf{F}}\left(  u\right)  dx.$ By using the well-known Karamata theorem
(see, for instance, Theorem B.1.5 in \citeauthor{deHF06},
\citeyear[page 363]{deHF06}) we have, for large $u,$%
\[
\Pi\sim\frac{\gamma_{1}}{1-\gamma_{1}}u\overline{\mathbf{F}}\left(  u\right)
,\text{ }0<\gamma_{1}<1.
\]
As we see, a semi-parametric estimator for\textbf{ }$\mathbf{F}$ is needed in
order to estimate the premium $\Pi.$ To this end, let us define%
\[
C\left(  x\right)  :=\mathbf{P}\left(  \mathbf{X}\leq x\leq\mathbf{Y\mid
X}\leq\mathbf{Y}\right)  =\mathbf{P}\left(  X\leq x\leq Y\right)  ,
\]
with $\mathbf{Y}$ being the truncation rv introduced in Section \ref{sec1}.
This quantity $C$ is very crucial as it plays a prominent role is the
statistical inference under random truncation. In other words, we have%
\[
C\left(  x\right)  =p^{-1}\mathbf{F}\left(  x\right)  \overline{\mathbf{G}%
}\left(  x\right)  =F\left(  x\right)  -G\left(  x\right)  =\overline
{G}\left(  x\right)  -\overline{F}\left(  x\right)  .
\]
It is worth mentioning that, since $\mathbf{F}$ and $\mathbf{G}$ are
heavy-tailed then their right endpoints are infinite and thus they are equal.
Therefore, from \cite{Wood85}, the functions $\mathbf{F},$ $F$ and $C$ are
linked by%
\[
C\left(  x\right)  d\mathbf{F}\left(  x\right)  =\mathbf{F}\left(  x\right)
dF\left(  x\right)  ,
\]
known as self-consistency equation\citep[see, e.g.,][]{SS09}, whose solution
is%
\begin{equation}
\mathbf{F}\left(  x\right)  =\exp-\Lambda\left(  x\right)  , \label{F}%
\end{equation}
where $\Lambda\left(  x\right)  :=\int_{x}^{\infty}dF\left(  z\right)
/C\left(  z\right)  .$ Replacing $F$ and $C$ by their respective empirical
counterparts $F_{n}\left(  x\right)  :=n^{-1}\sum_{i=1}^{n}\mathbf{1}\left(
X_{i}\leq x\right)  $ (the usual empirical df based on the fully observed
sample $\left(  X_{1},...,X_{n}\right)  $) and $C_{n}\left(  x\right)
:=n^{-1}\sum\limits_{i=1}^{n}\mathbf{1}\left(  X_{i}\leq x\leq Y_{i}\right)
,$ yields the well-known Lynden-Bell product limit estimator
\citep{LB-71}\textbf{\ }of $\mathbf{F,}$%
\begin{equation}
\mathbf{F}_{n}\left(  x\right)  =\exp-\Lambda_{n}\left(  x\right)  ,
\label{Fn}%
\end{equation}
where $\Lambda_{n}\left(  x\right)  :=\int_{x}^{\infty}dF_{n}\left(  z\right)
/C_{n}\left(  z\right)  .$ If there are no ties, $\mathbf{F}_{n}$ may be put
in the form%
\begin{equation}
\mathbf{F}_{n}(x):=\prod_{X_{i:n}>x}\left(  1-\frac{1}{nC_{n}\left(
X_{i:n}\right)  }\right)  . \label{Lynden}%
\end{equation}
Since $\overline{\mathbf{F}}$ is regularly varying at infinity with index
$-1/\gamma_{1},$ then
\[
\overline{\mathbf{F}}\left(  x\right)  \sim\overline{\mathbf{F}}\left(
U\left(  n/k\right)  \right)  \left(  x/U\left(  n/k\right)  \right)
^{-1/\gamma_{1}},\text{ as }x\rightarrow\infty.
\]
This leads us to derive a Weissman-type estimator \citep{Weiss-78}%
\[
\widehat{\overline{\mathbf{F}}}\left(  x\right)  =\left(  \frac{x}{X_{n-k:n}%
}\right)  ^{-1/\widehat{\gamma}_{1}}\overline{\mathbf{F}}_{n}\left(
X_{n-k:n}\right)  ,
\]
for the distribution tail $\overline{\mathbf{F}}$ with truncated data. Note
that%
\[
\mathbf{F}_{n}(X_{n-k:n})=%
{\displaystyle\prod\limits_{i=n-k+1}^{n}}
\left(  1-\frac{1}{nC_{n}\left(  X_{i:n}\right)  }\right)  .
\]
Thus, the distribution tail estimator is of the form%
\[
\widehat{\overline{\mathbf{F}}}\left(  x\right)  :=\left(  \frac{x}{X_{n-k:n}%
}\right)  ^{-1/\widehat{\gamma}_{1}}\left\{  1-%
{\displaystyle\prod\limits_{i=1}^{k}}
\left(  1-\frac{1}{nC_{n}\left(  X_{n-i+1:n}\right)  }\right)  \right\}  .
\]
Consequently, we define an estimator $\widehat{\Pi}_{n}$ to the premium $\Pi$
as follows:%
\[
\widehat{\Pi}_{n}:=\frac{\widehat{\gamma}_{1}}{1-\widehat{\gamma}_{1}%
}X_{n-k:n}\left(  \frac{u}{X_{n-k:n}}\right)  ^{1-1/\widehat{\gamma}_{1}%
}\left\{  1-%
{\displaystyle\prod\limits_{i=1}^{k}}
\left(  1-\frac{1}{nC_{n}\left(  X_{n-i+1:n}\right)  }\right)  \right\}  .
\]
This estimator coincides with that proposed and applied to the Norwegian fire
data by \cite{BMD}, in the non truncation case. Prior to establish the
asymptotic normality of $\widehat{\Pi}_{n}$ (Theorem \ref{Theorem3}), we give,
in the following basic result, an asymptotic representation to the Lynden-bell
estimator $\mathbf{F}_{n}$ (in $X_{n-k:n}).$ This result will of prime
importance in the study of the limiting behaviors of many statistics based on
truncated data exhibiting extreme values.

\begin{theorem}
\label{Theorem2}Assume that the second-order conditions of regular variation
$(\ref{second-order})$ hold with $\gamma_{1}<\gamma_{2}.\ $Let $k:=k_{n}$ be a
sequence of integers such that $k\rightarrow\infty,$ $k/n\rightarrow0.$ Then%
\[
\sqrt{k}\left(  \frac{\overline{\mathbf{F}}_{n}\left(  X_{n-k:n}\right)
}{\overline{\mathbf{F}}\left(  X_{n-k:n}\right)  }-1\right)  =\frac{\gamma
_{1}\gamma_{2}}{\left(  \gamma_{1}+\gamma_{2}\right)  ^{2}}\int_{0}%
^{1}s^{-\gamma/\gamma_{2}-1}W_{1}\left(  s\right)  ds+\frac{\gamma_{2}}%
{\gamma_{1}+\gamma_{2}}W_{1}\left(  1\right)  +o_{\mathbf{p}}\left(  1\right)
.
\]
Consequently,%
\[
\sqrt{k}\left(  \frac{\overline{\mathbf{F}}_{n}\left(  X_{n-k:n}\right)
}{\overline{\mathbf{F}}\left(  X_{n-k:n}\right)  }-1\right)  \overset
{\mathcal{D}}{\rightarrow}\mathcal{N}\left(  0,\frac{\gamma_{2}^{2}}%
{\gamma_{2}^{2}-\gamma_{1}^{2}}\right)  ,\text{ as }n\rightarrow\infty.
\]

\end{theorem}

\begin{remark}
\label{Rmrk}Under the assumptions of Theorem \ref{Theorem2}, we
have$\mathbf{\ }$
\[
\frac{\overline{\mathbf{F}}_{n}\left(  X_{n-k:n}\right)  }{\overline
{\mathbf{F}}\left(  X_{n-k:n}\right)  }\overset{\mathbf{p}}{\rightarrow
}1,\text{ as }n\rightarrow\infty.
\]

\end{remark}

\noindent To establish the asymptotic normality $\widehat{\Pi}_{n},$ we
require the second-order regular variation to $\mathbf{F.}$ That is, we
suppose that%
\begin{equation}
\underset{t\rightarrow\infty}{\lim}\dfrac{1}{\mathbf{A}\left(  t\right)
}\left(  \dfrac{\overline{\mathbf{F}}\left(  tx\right)  }{\overline
{\mathbf{F}}\left(  t\right)  }-x^{-1/\gamma_{1}}\right)  =x^{-1/\gamma_{1}%
}\dfrac{x^{\tau_{1}/\gamma_{1}}-1}{\tau_{1}\gamma_{1}},
\label{secon-orderFbold}%
\end{equation}
for any $x>0,$ where $\left\vert \mathbf{A}\right\vert $ is some regularly
varying function at infinity with index $\tau_{1}/\gamma_{1},$ where $\tau
_{1}<0$ is the second-order parameter. \ For asymptotic theory requirements,
one has to specify the relation between the retention level $u$ and the
quantile $U\left(  n/k\right)  .$ Indeed, as mentioned in \cite{VanBer},
amongst others, extreme value methodology typically applies to $u$ values for
which $\mathbf{P}(\mathbf{X}>u)=O(1/n),$ hence $\mathbf{P}(X>u)=O(1/n).$ This
leads to situate $u=u_{n}$ with respect to $U\left(  n/k\right)  $ so that,
for large\textbf{ }$n,$ the quotient $u/U\left(  n/k\right)  $ tends to some
constant $a.$

\begin{theorem}
\label{Theorem3}Assume that the second-order regular variation conditions
$(\ref{secon-orderFbold})$ hold with $0<\gamma_{1}<1$ and $\gamma_{1}%
<\gamma_{2}.$ Let $k:=k_{n}$ be a sequence of integers such that
$k\rightarrow\infty,$ $k/n\rightarrow0$ and $\sqrt{k}\mathbf{A}\left(
U\left(  n/k\right)  \right)  \rightarrow\lambda^{\ast}<\infty.$ Then,
whenever $u/U\left(  n/k\right)  \rightarrow a,$ we have as $n\rightarrow
\infty,$%
\[
\frac{\sqrt{k}\left(  \widehat{\Pi}_{n}-\Pi\right)  }{\left(  u/U\left(
n/k\right)  \right)  ^{1-1/\gamma_{1}}U\left(  n/k\right)  \overline
{\mathbf{F}}\left(  U\left(  n/k\right)  \right)  }\overset{\mathcal{D}%
}{\rightarrow}\mathcal{N}\left(  \frac{\lambda^{\ast}}{\left(  \gamma
_{1}-1-\tau_{1}\right)  \left(  \gamma_{1}-1\right)  },\sigma^{\ast2}\right)
,
\]
where%
\[
\sigma^{\ast2}:=\zeta^{2}\sigma^{2}+\frac{\gamma_{1}^{2}\gamma_{2}^{2}%
}{\left(  \gamma_{2}^{2}-\gamma_{1}^{2}\right)  \left(  1-\gamma_{1}\right)
^{2}}+2\frac{\gamma\gamma_{1}\zeta\delta^{\ast}}{1-\gamma_{1}},
\]
with $\sigma^{2}$ as defined in Corollary \ref{cor1}, $\zeta:=\left(  \left(
1-\gamma_{1}\right)  \log a+\gamma_{1}\right)  /\left(  \gamma_{1}\left(
1-\gamma_{1}\right)  ^{2}\right)  $ and%
\begin{align*}
\delta^{\ast} &  :=\dfrac{c}{\gamma\gamma_{2}}+\dfrac{c_{2}}{\gamma_{1}%
}\left(  R\left(  1,1\right)  -\int_{0}^{1}\dfrac{R\left(  1,t\right)  }%
{t}dt\right)  \\
&  +\dfrac{c_{2}}{\gamma_{1}+\gamma_{2}}\left(  \int_{0}^{1}\dfrac{R\left(
s,1\right)  }{s^{\gamma/\gamma_{2}+1}}ds-%
{\displaystyle\int_{0}^{1}}
{\displaystyle\int_{0}^{1}}
\dfrac{R\left(  s,t\right)  }{ts^{^{\gamma/\gamma_{2}+1}}}dsdt\right)  .
\end{align*}

\end{theorem}

\section{\textbf{Proofs\label{sec5}}}

\subsection{Proof of Theorem $\ref{Theorem1}$}

We begin by a brief introduction on the weak approximation of a weighed tail
copula process given in Proposition 1 of \cite{EHL-2006}. Set $U_{i}%
:=\overline{F}\left(  X_{i}\right)  $ and $V_{i}:=\overline{G}\left(
Y_{i}\right)  ,$ for $i=1,...,n,$ and let $C\left(  x,y\right)  $ be the joint
df of $\left(  U_{i},V_{i}\right)  .$ The copula function $C$ and its
corresponding tail $R,$ defined in $(\ref{R}),$ are linked by $C\left(
tx,ty\right)  -R\left(  x,y\right)  =O\left(  t^{\epsilon}\right)  ,$ as
$t\downarrow0,$ for some $\epsilon>0,$ uniformly for $x,y\geq0$ and
$\max\left(  x,y\right)  \leq1$ \citep{Huang-1992}$.$ Let us define%
\[
\upsilon_{n}\left(  x,y\right)  :=\sqrt{k}\left(  \mathbf{T}_{n}\left(
x,y\right)  -R_{n}\left(  x,y\right)  \right)  ,\text{ }x,y>0,
\]
where%
\[
\mathbf{T}_{n}\left(  x,y\right)  :=\frac{1}{k}\sum_{i=1}^{n}\mathbf{1}\left(
U_{i}<\frac{k}{n}x,\text{ }V_{i}<\frac{k}{n}y\right)  \text{ and }R_{n}\left(
x,y\right)  :=\frac{n}{k}C\left(  \frac{kx}{n},\frac{ky}{n}\right)  .
\]
In the sequel, we will need the following two empirical processes:%
\[
\alpha_{n}\left(  x\right)  :=\upsilon_{n}\left(  x,\infty\right)  =\sqrt
{k}\left(  \mathbf{U}_{n}\left(  x\right)  -x\right)  \text{ and }\beta
_{n}\left(  y\right)  :=\upsilon_{n}\left(  \infty,y\right)  =\sqrt{k}\left(
\mathbf{V}_{n}\left(  y\right)  -y\right)  ,
\]
where%
\[
\mathbf{U}_{n}\left(  x\right)  :=\mathbf{T}_{n}\left(  x,\infty\right)
=\frac{1}{k}\sum_{i=1}^{n}\mathbf{1}\left(  U_{i}<\frac{k}{n}x\right)  ,
\]
and%
\[
\mathbf{V}_{n}\left(  y\right)  :=\mathbf{T}_{n}\left(  \infty,y\right)
=\frac{1}{k}\sum_{i=1}^{n}\mathbf{1}\left(  V_{i}<\frac{k}{n}y\right)  .
\]
From assertions $\left(  3.8\right)  $ and $\left(  3.9\right)  $ in
\cite{EHL-2006}, there exists a Gaussian process $W_{R}\left(  x,y\right)  ,$
defined on the probability space $\left(  \Omega,\mathcal{A},\mathbf{P}%
\right)  ,$ with mean zero and covariance%
\begin{equation}
\mathbf{E}\left[  W_{R}\left(  x_{1},y_{1}\right)  W_{R}\left(  x_{2}%
,y_{2}\right)  \right]  =R(\min\left(  x_{1},x_{2}\right)  ,\min\left(
y_{1},y_{2}\right)  ), \label{cov}%
\end{equation}
such that for any $M>0$%
\[
\sup\limits_{0<x,y\leq M}\dfrac{\left\vert \upsilon_{n}\left(  x,y\right)
-W_{R}\left(  x,y\right)  \right\vert }{\left\{  \max\left(  x,y\right)
\right\}  ^{\eta}}=o_{\mathbf{p}}\left(  1\right)  ,
\]
and%
\begin{equation}
\sup\limits_{0<x\leq M}\dfrac{\left\vert \alpha_{n}\left(  x\right)
-W_{1}\left(  x\right)  \right\vert }{x^{\eta}}=o_{\mathbf{p}}\left(
1\right)  =\sup\limits_{0<y\leq M}\dfrac{\left\vert \beta_{n}\left(  y\right)
-W_{2}\left(  y\right)  \right\vert }{y^{\eta}}, \label{approx}%
\end{equation}
as $n\rightarrow\infty,$ for any $0\leq\eta<1/2,$ where
\[
W_{1}\left(  x\right)  :=W_{R}\left(  x,\infty\right)  \text{ and }%
W_{2}\left(  y\right)  :=W_{R}\left(  \infty,y\right)  ,
\]
are two standard Wiener processes such that $\mathbf{E}\left[  W_{1}\left(
x\right)  W_{2}\left(  y\right)  \right]  =R\left(  x,y\right)  .\medskip$

\noindent To prove our result, we will write the tail index estimator
$\widehat{\gamma}_{1}$ in terms of the processes $\alpha_{n}\left(
\cdot\right)  $ and $\beta_{n}\left(  \cdot\right)  .$ We start by splitting
$\widehat{\gamma}_{1}-\gamma_{1}$ into the sum of two terms%
\[
T_{n1}:=\frac{\widehat{\gamma}_{2}\left(  \gamma_{2}-\gamma\right)
+\gamma_{2}\gamma}{\left(  \widehat{\gamma}_{2}-\widehat{\gamma}\right)
\left(  \gamma_{2}-\gamma\right)  }\left(  \widehat{\gamma}-\gamma\right)
\text{ and }T_{n2}:=-\frac{\gamma^{2}}{\left(  \widehat{\gamma}_{2}%
-\widehat{\gamma}\right)  \left(  \gamma_{2}-\gamma\right)  }\left(
\widehat{\gamma}_{2}-\gamma_{2}\right)  .
\]
Note that, for two sequences of rv's $V_{n}^{\left(  1\right)  }$ and
$V_{n}^{\left(  2\right)  },$ we use the notation$\ V_{n}^{\left(  1\right)
}\approx V_{n}^{\left(  2\right)  }$ to say that$V_{n}^{\left(  1\right)
}=V_{n}^{\left(  2\right)  }\left(  1+o_{\mathbf{p}}\left(  1\right)  \right)
,$\ as $n\rightarrow\infty.$ Since both $\widehat{\gamma}$ and $\widehat
{\gamma}_{2}$ are consistent estimators \citep[]{Mas82}, then, as
$n\rightarrow\infty,$ we have%
\[
T_{n1}\approx\frac{c}{\gamma}\left(  \widehat{\gamma}-\gamma\right)  \text{
and }T_{n2}\approx-\frac{c_{2}}{\gamma_{2}}\left(  \widehat{\gamma}_{2}%
-\gamma_{2}\right)  ,
\]
where $c_{1}$ and $c_{2}$ are those defined in Theorem $\ref{Theorem1}.$ In
other words, we have, as $n\rightarrow\infty,$%
\begin{equation}
\sqrt{k}\left(  \widehat{\gamma}_{1}-\gamma_{1}\right)  \approx\frac{c}%
{\gamma}\sqrt{k}\left(  \widehat{\gamma}-\gamma\right)  -\frac{c_{2}}%
{\gamma_{2}}\sqrt{k}\left(  \widehat{\gamma}_{2}-\gamma_{2}\right)
\label{gam-1}%
\end{equation}
Next, we represent $\sqrt{k}\left(  \widehat{\gamma}-\gamma\right)  $ and
$\sqrt{k}\left(  \widehat{\gamma}_{2}-\gamma_{2}\right)  $ in terms of
$\alpha_{n}\left(  \cdot\right)  $ and $\beta_{n}\left(  \cdot\right)  $
respectively. For the first term, we use the first-order condition of regular
variation of $\overline{F}$ $(\ref{RV-2})$ and apply Theorem 1.2.2 in
\cite{deHF06} to have%
\[
\lim_{n\rightarrow\infty}\frac{n}{k}\int_{F^{\leftarrow}\left(  1-k/n\right)
}^{\infty}t^{-1}\overline{F}\left(  t\right)  dt=\gamma,
\]
this allows us to write $\widehat{\gamma}=\dfrac{n}{k}\int_{X_{n-k:n}}%
^{\infty}t^{-1}\overline{F}_{n}\left(  t\right)  dt.$ Now, we consider the
following decomposition $\widehat{\gamma}-\gamma=S_{n1}+S_{n2}+S_{n3},$ where%
\[
S_{n1}:=\frac{n}{k}\int_{X_{n-k:n}}^{\infty}t^{-1}\left(  \overline{F}%
_{n}\left(  t\right)  -\overline{F}\left(  t\right)  \right)  dt,\text{
}S_{n2}:=-\frac{n}{k}\int_{F^{-1}\left(  1-k/n\right)  }^{X_{n-k:n}}%
t^{-1}\overline{F}\left(  t\right)  dt
\]
and%
\[
S_{n3}:=\frac{n}{k}\int_{F^{-1}\left(  1-k/n\right)  }^{\infty}t^{-1}%
\overline{F}\left(  t\right)  dt-\gamma.
\]
It is easy to verify that, almost surely, we have%
\begin{equation}
\overline{F}_{n}\left(  t\right)  =\dfrac{k}{n}\mathbf{U}_{n}\left(  \dfrac
{n}{k}\overline{F}\left(  t\right)  \right)  . \label{rep}%
\end{equation}
Without loss of generality and after two successive changes of variables
$(u=tX_{n-k:n}$ then $s=n\overline{F}\left(  tX_{n-k:n}\right)  /k),$ we have%
\[
S_{n1}=\int_{\frac{n}{k}\overline{F}\left(  X_{n-k:n}\right)  }^{0}%
\frac{\mathbf{U}_{n}\left(  s\right)  -s}{F^{\leftarrow}\left(  1-sk/n\right)
}dF^{\leftarrow}\left(  1-sk/n\right)  ,
\]
which we decompose into%
\begin{align*}
S_{n1}  &  =\int_{1}^{0}\frac{\mathbf{U}_{n}\left(  s\right)  -s}%
{F^{\leftarrow}\left(  1-sk/n\right)  }dF^{\leftarrow}\left(  1-sk/n\right) \\
&  \ \ \ \ \ \ \ +\int_{\frac{n}{k}\overline{F}\left(  X_{n-k:n}\right)  }%
^{1}\frac{\mathbf{U}_{n}\left(  s\right)  -s}{F^{\leftarrow}\left(
1-sk/n\right)  }dF^{\leftarrow}\left(  1-sk/n\right)  .
\end{align*}
For the purpose of using Potter's result of Lemma \ref{Lem3} for the quantile
function $s\rightarrow F^{\leftarrow}\left(  1-s\right)  ,$ we write%
\begin{align*}
S_{n1}  &  =\int_{1}^{0}\frac{F^{\leftarrow}\left(  1-k/n\right)
}{F^{\leftarrow}\left(  1-sk/n\right)  }\left(  \mathbf{U}_{n}\left(
s\right)  -s\right)  d\frac{F^{\leftarrow}\left(  1-sk/n\right)
}{F^{\leftarrow}\left(  1-k/n\right)  }\\
&  +\int_{\frac{n}{k}\overline{F}\left(  X_{n-k:n}\right)  }^{1}%
\frac{F^{\leftarrow}\left(  1-k/n\right)  }{F^{\leftarrow}\left(
1-sk/n\right)  }\left(  \mathbf{U}_{n}\left(  s\right)  -s\right)
d\frac{F^{\leftarrow}\left(  1-sk/n\right)  }{F^{\leftarrow}\left(
1-k/n\right)  }.
\end{align*}
This allows us to write%
\[
S_{n1}\approx\gamma\left\{  \int_{0}^{1}s^{-1}\left(  \mathbf{U}_{n}\left(
s\right)  -s\right)  ds-\int_{\frac{n}{k}\overline{F}\left(  X_{n-k:n}\right)
}^{1}s^{-1}\left(  \mathbf{U}_{n}\left(  s\right)  -s\right)  ds\right\}  .
\]
In other words, we have, as $n\rightarrow\infty,$%
\begin{equation}
\sqrt{k}S_{n1}\approx\gamma\left\{  \int_{0}^{1}s^{-1}\alpha_{n}\left(
s\right)  ds-\int_{\frac{n}{k}\overline{F}\left(  X_{n-k:n}\right)  }%
^{1}s^{-1}\alpha_{n}\left(  s\right)  ds\right\}  . \label{sn-1}%
\end{equation}
As for the second term $S_{n2},$ we use the mean value theorem to get%
\[
S_{n2}=-\frac{n}{k}\left(  X_{n-k:n}-U\left(  n/k\right)  \right)  z_{n}%
^{-1}\overline{F}\left(  z_{n}\right)  ,
\]
where $z_{n}$ is a sequence of rv's lying between $X_{n-k:n}$ and $U\left(
n/k\right)  .$\ Observe that we have%
\[
S_{n2}=-\frac{\overline{F}\left(  z_{n}\right)  }{\overline{F}\left(
F^{\leftarrow}\left(  1-k/n\right)  \right)  }\frac{U\left(  n/k\right)
}{z_{n}}\left(  \frac{X_{n-k:n}}{U\left(  n/k\right)  }-1\right)  .
\]
Since $X_{n-k:n}/U\left(  n/k\right)  \overset{\mathbf{p}}{\rightarrow}1,$
then $z_{n}/U\left(  n/k\right)  \overset{\mathbf{p}}{\rightarrow}1$ and
$\dfrac{n}{k}\overline{F}\left(  z_{n}\right)  \overset{\mathbf{p}%
}{\rightarrow}1.$ It follows that%
\[
S_{n2}\approx-\left(  \frac{X_{n-k:n}}{U\left(  n/k\right)  }-1\right)  .
\]
Recall that $U_{i}=\overline{F}\left(  X_{i}\right)  $ and note that
$U_{i:n}=\overline{F}\left(  X_{n-i+1:n}\right)  ,$ therefore%
\[
S_{n2}\approx-\left(  \frac{F^{\leftarrow}\left(  1-U_{k+1:n}\right)
}{U\left(  n/k\right)  }-1\right)  .
\]
We use Potter's bound inequalities (see Lemma \ref{Lem3}) together with the
mean value theorem to write $S_{n2}\approx\gamma\left(  \dfrac{n}{k}%
U_{k+1:n}-1\right)  .$ Since $\mathbf{U}_{n}\left(  \dfrac{n}{k}%
U_{k+1:n}\right)  =1,$ then%
\begin{equation}
\sqrt{k}S_{n2}\approx-\gamma\alpha_{n}\left(  \frac{n}{k}U_{k+1:n}\right)  .
\label{sn-2}%
\end{equation}
By summing up $\left(  \ref{sn-1}\right)  $ and $\left(  \ref{sn-2}\right)  $
and making use of the weak approximation $(\ref{approx})$ for $\alpha
_{n}\left(  \cdot\right)  ,$ we get%
\begin{align*}
&  \sqrt{k}\left(  S_{n1}+S_{n2}\right) \\
&  \approx\gamma\left\{  \int_{0}^{1}s^{-1}W_{1}\left(  s\right)
ds-\int_{\frac{n}{k}\overline{F}\left(  X_{n-k:n}\right)  }^{1}s^{-1}%
W_{1}\left(  s\right)  ds-W_{1}\left(  \frac{n}{k}U_{k+1:n}\right)  \right\}
.
\end{align*}
Next, we show that $I_{n}:=\int_{\frac{n}{k}\overline{F}\left(  X_{n-k:n}%
\right)  }^{1}s^{-1}W_{1}\left(  s\right)  ds\overset{\mathbf{p}}{\rightarrow
}0.$ For arbitrary $\epsilon,\vartheta>0,$ we write
\[
\mathbf{P}\left(  \left\vert I_{n}\right\vert >\vartheta\right)
\leq\mathbf{P}\left(  \int_{1-\epsilon}^{1}s^{-1}W_{1}\left(  s\right)
ds>\vartheta\right)  +\mathbf{P}\left(  \left\vert \frac{n}{k}\overline
{F}\left(  X_{n-k:n}\right)  -1\right\vert >\epsilon\right)  .
\]
The fact that $\mathbf{E}\left\vert W_{1}\left(  s\right)  \right\vert \leq
s^{1/2}$ implies that $\mathbf{E}\left\vert \int_{1-\epsilon}^{1}s^{-1}%
W_{1}\left(  s\right)  ds\right\vert \leq\sqrt{2}\left(  1-\left(
1-\epsilon\right)  ^{1/2}\right)  .$ Therefore, by Chebyshev's inequality, we
infer that%
\[
\mathbf{P}\left(  \int_{1-\epsilon}^{1}s^{-1}W_{1}\left(  s\right)
ds>\vartheta\right)  \leq\sqrt{2}\vartheta^{-2}\left(  1-\left(
1-\epsilon\right)  ^{1/2}\right)  .
\]
On the other hand, we have $\dfrac{n}{k}\overline{F}\left(  X_{n-k:n}\right)
\overset{\mathbf{p}}{\rightarrow}1$ this means that for all large $n,$%
\[
\mathbf{P}\left(  \left\vert \frac{n}{k}\overline{F}\left(  X_{n-k:n}\right)
-1\right\vert >\epsilon\right)  \leq\sqrt{2}\epsilon^{-2}\left(  1-\left(
1-\epsilon\right)  ^{1/2}\right)  .
\]
It follows that $\mathbf{P}\left(  \left\vert I_{n}\right\vert >\vartheta
\right)  \leq\sqrt{2}\left(  \vartheta^{-2}+\epsilon^{-2}\right)  \left(
1-\left(  1-\epsilon\right)  ^{1/2}\right)  $ which tends to zero when
$\epsilon,\vartheta\downarrow0,$ as sought. Since $nU_{k+1:n}/k\overset
{\mathbf{p}}{\rightarrow}1,$ then by using similar arguments as the above we
have $W_{1}\left(  nU_{k+1:n}/k\right)  =W_{1}\left(  1\right)  +o_{\mathbf{p}%
}\left(  1\right)  .$ Consequently, we have%
\[
\sqrt{k}\left(  S_{n1}+S_{n2}\right)  =\left(  1+o_{\mathbf{p}}\left(
1\right)  \right)  \gamma\int_{0}^{1}s^{-1}W_{1}\left(  s\right)  ds-\gamma
W_{1}\left(  1\right)  .
\]
For the third term, it suffices to use the second-order condition of regular
variation to obtain
\[
\sqrt{k}S_{n3}=\frac{\sqrt{k}A\left(  n/k\right)  }{1-\tau}\left(  1+o\left(
1\right)  \right)  \text{ as }n\rightarrow\infty.
\]
In summary, we have%
\begin{equation}
\sqrt{k}\left(  \widehat{\gamma}-\gamma\right)  =\gamma\int_{0}^{1}s^{-1}%
W_{1}\left(  s\right)  ds-\gamma W_{1}\left(  1\right)  +\frac{\sqrt
{k}A\left(  n/k\right)  }{1-\tau}+o_{\mathbf{p}}\left(  1\right)  .
\label{gam}%
\end{equation}
Likewise, we write $\widehat{\gamma}_{2}=\dfrac{n}{k}\int_{Y_{n-k:n}}^{\infty
}t^{-1}\overline{G}_{n}\left(  t\right)  dt,$ where $G_{n}\left(  x\right)
:=n^{-1}\sum_{i=1}^{n}\mathbf{1}\left(  Y_{i}\leq x\right)  $ is the usual
empirical df based on the fully observed sample $\left(  Y_{1},...,Y_{n}%
\right)  .$ Then, by using similar arguments, we express $\widehat{\gamma}%
_{2}$ in terms of the process $\beta_{n}\left(  \cdot\right)  $ as follows:%
\begin{align*}
&  \sqrt{k}\left(  \widehat{\gamma}_{2}-\gamma_{2}\right) \\
&  \approx\gamma_{2}\left\{  \int_{0}^{1}s^{-1}\beta_{n}\left(  s\right)
ds-\int_{\frac{n}{k}\overline{G}\left(  Y_{n-k:n}\right)  }^{1}s^{-1}\beta
_{n}\left(  s\right)  ds-\beta_{n}\left(  \frac{n}{k}V_{k+1:n}\right)
\right\}  .
\end{align*}
Then by using approximation $(\ref{approx})$ for $\beta_{n}\left(
\cdot\right)  ,$ we obtain%
\begin{equation}
\sqrt{k}\left(  \widehat{\gamma}_{2}-\gamma_{2}\right)  =\gamma_{2}\int
_{0}^{1}s^{-1}W_{2}\left(  s\right)  ds-\gamma_{2}W_{2}\left(  1\right)
+\frac{\sqrt{k}A_{2}\left(  n/k\right)  }{1-\tau_{2}}+o_{\mathbf{p}}\left(
1\right)  . \label{gam-2}%
\end{equation}
Finally, substituting results $\left(  \ref{gam}\right)  $ and $\left(
\ref{gam-2}\right)  $ in equation $\left(  \ref{gam-1}\right)  $ achieves the
proof.$\hfill\square$

\subsection{Proof of Corollary $\ref{cor1}$}

Elementary calculations, using the covariance formula $(\ref{cov})$ and the
fact that $\mathbf{E}\left[  \int_{0}^{1}s^{-1}W_{i}\left(  s\right)
ds\right]  ^{2}=2,$ $i=1,2,$ straightforwardly lead to the result.$\hfill
\square$

\subsection{Proof of Corollary $\ref{cor2}$}

It suffices to plug the estimate of each parameter in the result of Corollary
$\ref{cor1}.$ To estimate the limits $\lambda$ and $\lambda_{2},$ we exploit
the second-order conditions of regular variation $\left(  \ref{second-order}%
\right)  .$ We have, as $z\rightarrow\infty,$%
\[
A\left(  z\right)  \sim\tau\frac{U\left(  zx\right)  /U\left(  z\right)
-x^{\gamma}}{x^{\gamma}\left(  x^{\tau}-1\right)  },\text{ for any }x>0.
\]
In particular, for $x=1/2,$ and $z=n/k,$ we have%
\[
A\left(  n/k\right)  \sim\tau\frac{U\left(  \dfrac{n}{2k}\right)  /U\left(
\dfrac{n}{k}\right)  -2^{-\gamma}}{2^{-\gamma}\left(  2^{-\tau}-1\right)  }.
\]
Hence, we take%
\[
\widehat{A}\left(  n/k\right)  =\widehat{\tau}\frac{X_{n-2k:n}/X_{n-k:n}%
-2^{-\widehat{\gamma}}}{2^{-\widehat{\gamma}}\left(  2^{-\widehat{\tau}%
}-1\right)  }=\widehat{\tau}\frac{X_{n-2k:n}-2^{-\widehat{\gamma}}X_{n-k:n}%
}{2^{-\widehat{\gamma}}\left(  2^{-\widehat{\tau}}-1\right)  X_{n-k:n}},
\]
an estimate of $A\left(  n/k\right)  .$ Thus, the expression of $\widehat
{\lambda}$ readily follows. The same idea applies to $\lambda_{2}$ as
well.$\hfill\square$

\subsection{Proof of Theorem $\ref{Theorem2}$}

For convenience we set%
\[
\mathbf{D}_{n}^{\ast}:=\frac{\overline{\mathbf{F}}_{n}\left(  X_{n-k:n}%
\right)  -\overline{\mathbf{F}}\left(  X_{n-k:n}\right)  }{\overline
{\mathbf{F}}\left(  X_{n-k:n}\right)  }.
\]
Since $\overline{\mathbf{F}}$ is regularly varying at infinity with index
$-1/\gamma_{1}$ and $X_{n-k:n}/U\left(  n/k\right)  \overset{\mathbf{P}%
}{\rightarrow}1,$ then $\overline{\mathbf{F}}\left(  X_{n-k:n}\right)
\approx\overline{\mathbf{F}}\left(  U\left(  n/k\right)  \right)  $ and
therefore%
\[
\mathbf{D}_{n}^{\ast}\approx\frac{\overline{\mathbf{F}}_{n}\left(
X_{n-k:n}\right)  -\overline{\mathbf{F}}\left(  X_{n-k:n}\right)  }%
{\overline{\mathbf{F}}\left(  U\left(  n/k\right)  \right)  }.
\]
Using equations $\left(  \ref{F}\right)  $ and $\left(  \ref{Fn}\right)  ,$ we
have%
\[
\mathbf{D}_{n}^{\ast}\approx\frac{\left\{  1-\exp-\Lambda_{n}\left(
X_{n-k:n}\right)  \right\}  -\left\{  1-\exp-\Lambda\left(  X_{n-k:n}\right)
\right\}  }{1-\exp-\Lambda\left(  U\left(  n/k\right)  \right)  }.
\]
Since both $\overline{\mathbf{F}}_{n}\left(  X_{n-k:n}\right)  $ and
$\overline{\mathbf{F}}\left(  X_{n-k:n}\right)  $ tend to zero in probability,
then $\Lambda_{n}\left(  X_{n-k:n}\right)  $ and $\Lambda\left(
X_{n-k:n}\right)  $ go to zero in probability as well.\ Hence, by using the
approximation $1-\exp(-x)\sim x,$ as $x\rightarrow0,$ we get%
\[
\mathbf{D}_{n}^{\ast}\approx\frac{\Lambda_{n}\left(  X_{n-k:n}\right)
-\Lambda\left(  X_{n-k:n}\right)  }{\Lambda\left(  U\left(  n/k\right)
\right)  }=:\mathbf{D}_{n}.
\]
Now, we study the asymptotic behavior of $\mathbf{D}_{n}.$ The numerator%
\[
\Lambda_{n}\left(  X_{n-k:n}\right)  -\Lambda\left(  X_{n-k:n}\right)
=-\int_{X_{n-k:n}}^{\infty}\frac{d\overline{F}_{n}\left(  z\right)  }%
{C_{n}\left(  z\right)  }+\int_{X_{n-k:n}}^{\infty}\frac{d\overline{F}\left(
z\right)  }{C\left(  z\right)  },
\]
may be decomposed into $S_{n1}+$ $S_{n2}+S_{n3},$ with%
\[
S_{n1}:=-\int_{U\left(  n/k\right)  }^{\infty}\frac{d\left(  \overline{F}%
_{n}\left(  z\right)  -\overline{F}\left(  z\right)  \right)  }{C\left(
z\right)  },\text{ }S_{n2}:=-\int_{X_{n-k:n}}^{\infty}\left\{  \frac{1}%
{C_{n}\left(  z\right)  }-\frac{1}{C\left(  z\right)  }\right\}  d\overline
{F}_{n}\left(  z\right)  ,
\]
and%
\[
S_{n3}:=-\int_{X_{n-k:n}}^{U\left(  n/k\right)  }\frac{d\left(  \overline
{F}_{n}\left(  z\right)  -\overline{F}\left(  z\right)  \right)  }{C\left(
z\right)  }.
\]
We will show that $\sqrt{k}S_{n1}/\Lambda\left(  U\left(  n/k\right)  \right)
$ is an asymptotically centred Gaussian rv while both $\sqrt{k}S_{n2}%
/\Lambda\left(  U\left(  n/k\right)  \right)  $ and $\sqrt{k}S_{n3}%
/\Lambda\left(  U\left(  n/k\right)  \right)  $ tend to zero (in probability)
as $n\rightarrow\infty.$ An integration by parts yields that $S_{n1}%
=S_{n1}^{\left(  1\right)  }-S_{n1}^{\left(  2\right)  },$ where%
\[
S_{n1}^{\left(  1\right)  }:=\frac{\overline{F}_{n}\left(  U\left(
n/k\right)  \right)  -k/n}{C\left(  U\left(  n/k\right)  \right)  }%
\]
and (with a change of variables)%
\[
S_{n1}^{\left(  2\right)  }:=\int_{1}^{\infty}\frac{\overline{F}_{n}\left(
zU\left(  n/k\right)  \right)  -\overline{F}\left(  zU\left(  n/k\right)
\right)  }{C^{2}\left(  zU\left(  n/k\right)  \right)  }dC\left(  zU\left(
n/k\right)  \right)  .
\]
It is easy to verify that%
\[
\frac{\sqrt{k}S_{n1}^{\left(  1\right)  }}{\Lambda\left(  U\left(  n/k\right)
\right)  }=\frac{k/n}{\Lambda\left(  U\left(  n/k\right)  \right)  C\left(
U\left(  n/k\right)  \right)  }\alpha_{n}\left(  1\right)  ,
\]
where $\alpha_{n}\left(  \cdot\right)  $ is the uniform tail empirical process
defined at the beginning of the proof of Theorem \ref{Theorem1}. From Lemma
$\left(  \ref{lem2}\right)  ,$\ we infer that%
\begin{equation}
\sqrt{k}S_{n1}^{\left(  1\right)  }/\Lambda\left(  U\left(  n/k\right)
\right)  \approx\gamma\gamma_{1}^{-1}\alpha_{n}\left(  1\right)  .
\label{Sn11}%
\end{equation}
For the term $S_{n1}^{\left(  2\right)  },$ we have%
\[
\frac{\sqrt{k}S_{n1}^{\left(  2\right)  }}{\Lambda\left(  U\left(  n/k\right)
\right)  }=\frac{\int_{1}^{\infty}\dfrac{C^{2}\left(  U\left(  n/k\right)
\right)  }{C^{2}\left(  zU\left(  n/k\right)  \right)  }\alpha_{n}\left(
\dfrac{n}{k}\overline{F}\left(  zU\left(  n/k\right)  \right)  \right)
d\dfrac{C\left(  zU\left(  n/k\right)  \right)  }{C\left(  U\left(
n/k\right)  \right)  }}{\left(  n/k\right)  \Lambda\left(  U\left(
n/k\right)  \right)  C\left(  U\left(  n/k\right)  \right)  }.
\]
From Lemma $\left(  \ref{lem1}\right)  $, we know that the function $C$ is
regularly varying at infinity with index $-1/\gamma_{2},$ then by using
Potter's inequality, together with $(\ref{limit}),$ we get%
\[
\frac{\sqrt{k}S_{n1}^{\left(  2\right)  }}{\Lambda\left(  U\left(  n/k\right)
\right)  }\approx-\left(  \gamma_{1}+\gamma_{2}\right)  ^{-1}\int_{1}^{\infty
}z^{1/\gamma_{2}-1}\alpha_{n}\left(  \frac{n}{k}\overline{F}\left(  zU\left(
n/k\right)  \right)  \right)  dz,
\]
which, by the change of variables $s=\dfrac{n}{k}\overline{F}\left(  zU\left(
n/k\right)  \right)  =\dfrac{n}{k}\overline{F}\left(  zF^{\leftarrow}\left(
1-k/n\right)  \right)  ,$ becomes%
\[
\frac{\sqrt{k}S_{n1}^{\left(  2\right)  }}{\Lambda\left(  U\left(  n/k\right)
\right)  }\approx\left(  \gamma_{1}+\gamma_{2}\right)  ^{-1}\int_{0}%
^{1}\left(  \psi_{n}\left(  s\right)  \right)  ^{1/\gamma_{2}-1}\alpha
_{n}\left(  s\right)  d\psi_{n}\left(  s\right)  ,
\]
where $\psi_{n}\left(  s\right)  :=F^{\leftarrow}\left(  1-ks/n\right)
/F^{\leftarrow}\left(  1-k/n\right)  .$ Making use, once again, of Potter's
inequality of Lemma \ref{Lem3} to the quantile function $s\rightarrow
F^{\leftarrow}\left(  1-s\right)  ,$ yields%
\begin{equation}
\frac{\sqrt{k}S_{n1}^{\left(  2\right)  }}{\Lambda\left(  U\left(  n/k\right)
\right)  }\approx-\frac{\gamma_{1}\gamma_{2}}{\left(  \gamma_{1}+\gamma
_{2}\right)  ^{2}}\int_{0}^{1}s^{-\gamma/\gamma_{2}-1}\alpha_{n}\left(
s\right)  ds. \label{Sn12}%
\end{equation}
Subtracting $\left(  \ref{Sn12}\right)  $ from $\left(  \ref{Sn11}\right)  $
and using the weak approximation\ $(\ref{approx}),$ we get%
\[
\frac{\sqrt{k}S_{n1}}{\Lambda\left(  U\left(  n/k\right)  \right)  }%
\approx\frac{\gamma_{1}\gamma_{2}}{\left(  \gamma_{1}+\gamma_{2}\right)  ^{2}%
}\int_{0}^{1}s^{-\gamma/\gamma_{2}-1}W_{1}\left(  s\right)  ds+\frac{\gamma
}{\gamma_{1}}W_{1}\left(  1\right)  +o_{p}\left(  1\right)  .
\]
Note that the centred rv $\int_{0}^{1}s^{-\gamma/\gamma_{2}-1}W_{1}\left(
s\right)  ds$ has a finite second moment (in fact it is equal to $2\gamma
_{2}^{2}/\left(  \left(  \gamma_{2}-\gamma\right)  \left(  \gamma_{2}%
-2\gamma\right)  \right)  .$\textbf{\ }As a result, the approximation above
becomes%
\[
\frac{\sqrt{k}S_{n1}}{\Lambda\left(  U\left(  n/k\right)  \right)  }%
=\frac{\gamma_{1}\gamma_{2}}{\left(  \gamma_{1}+\gamma_{2}\right)  ^{2}}%
\int_{0}^{1}s^{-\gamma/\gamma_{2}-1}W_{1}\left(  s\right)  ds+\frac{\gamma
}{\gamma_{1}}W_{1}\left(  1\right)  +o_{p}\left(  1\right)  .
\]
Now, we consider the second term $S_{n2}.\ $Since $\overline{F}_{n}\left(
z\right)  =0,$ for $z\geq X_{n:n},$ then%
\[
S_{n2}=\int_{X_{n-k:n}}^{X_{n:n}}\frac{C_{n}\left(  z\right)  -C\left(
z\right)  }{C_{n}\left(  z\right)  C\left(  z\right)  }d\overline{F}%
_{n}\left(  z\right)  .
\]
It follows that%
\[
\left\vert S_{n2}\right\vert \leq\theta_{n}\int_{X_{n-k:n}}^{\infty}%
\frac{\left\vert C_{n}\left(  z\right)  -C\left(  z\right)  \right\vert
}{C^{2}\left(  z\right)  }dF_{n}\left(  z\right)  ,
\]
where $\theta_{n}:=\sup_{X_{1:n}\leq z\leq X_{n:n}}\left\{  C\left(  z\right)
/C_{n}\left(  z\right)  \right\}  ,$ which is stochastically bounded
\citep[see, e.g.,][]{SW08}. We have $C=\overline{G}-\overline{F}$ and
$C_{n}=\overline{G}_{n}-\overline{F}_{n},$ then $\left\vert S_{n2}\right\vert
\leq\theta_{n}\left(  T_{n1}+T_{n2}\right)  ,$ where
\[
T_{n1}:=\int_{X_{n-k:n}}^{\infty}\dfrac{\left\vert \overline{F}_{n}\left(
z\right)  -\overline{F}\left(  z\right)  \right\vert }{C^{2}\left(  z\right)
}dF_{n}\left(  z\right)  \text{ and }T_{n2}:=\int_{X_{n-k:n}}^{\infty}%
\dfrac{\left\vert \overline{G}_{n}\left(  z\right)  -\overline{G}\left(
z\right)  \right\vert }{C^{2}\left(  z\right)  }dF_{n}\left(  z\right)  .
\]
\ The set $\mathcal{A}_{n,\epsilon}:=\left\{  \left\vert X_{n-k:n}/U\left(
n/k\right)  -1\right\vert >\epsilon\right\}  ,$ $0<\epsilon<1$ is such that
$\mathbf{P}\left(  \mathcal{A}_{n,\epsilon}\right)  \rightarrow0$ as
$n\rightarrow\infty.$ For convenience, let $u_{n,\epsilon}:=\left(
1-\epsilon\right)  U\left(  n/k\right)  $ and%
\[
T_{n1}\left(  \epsilon\right)  :=\int_{u_{n,\epsilon}}^{\infty}\dfrac
{\left\vert \overline{F}_{n}\left(  z\right)  -\overline{F}\left(  z\right)
\right\vert }{C^{2}\left(  z\right)  }dF_{n}\left(  z\right)  .
\]
It is obvious that, for $\vartheta>0,$
\[
\mathbf{P}\left(  \frac{\sqrt{k}T_{n1}}{\Lambda\left(  U\left(  n/k\right)
\right)  }>\vartheta\right)  \leq\mathbf{P}\left(  \frac{\sqrt{k}T_{n1}\left(
\epsilon\right)  }{\Lambda\left(  U\left(  n/k\right)  \right)  }%
>\vartheta\right)  +\mathbf{P}\left(  \mathcal{A}_{n}\right)  .
\]
Then it remains to show that $\mathbf{P}\left(  \dfrac{\sqrt{k}T_{n1}\left(
\epsilon\right)  }{\Lambda\left(  U\left(  n/k\right)  \right)  }%
>\vartheta\right)  \rightarrow0$ as $n\rightarrow\infty.\ $To this end, let us
write%
\begin{align*}
\frac{\sqrt{k}T_{n1}\left(  \epsilon\right)  }{\Lambda\left(  U\left(
n/k\right)  \right)  }  &  =\left\{  \frac{\overline{F}\left(  U\left(
n/k\right)  \right)  }{C\left(  U\left(  n/k\right)  \right)  }\right\}
\left\{  \frac{k/n}{\Lambda\left(  U\left(  n/k\right)  \right)  C\left(
U\left(  n/k\right)  \right)  }\right\} \\
&  \times\left\{  \frac{C\left(  U\left(  n/k\right)  \right)  }{C\left(
u_{n,\epsilon}\right)  }\right\}  ^{2}\int_{1}^{\infty}\dfrac{\left\vert
\alpha_{n}\left(  n\overline{F}\left(  zu_{n,\epsilon}\right)  /k\right)
\right\vert }{\left[  C\left(  zu_{n,\epsilon}\right)  /C\left(
u_{n,\epsilon}\right)  \right]  ^{2}}d\frac{F_{n}\left(  zu_{n,\epsilon
}\right)  }{\overline{F}\left(  U\left(  n/k\right)  \right)  }%
\end{align*}
The regular variation property of $C,$ that implies that $C\left(  U\left(
n/k\right)  \right)  /C\left(  u_{n,\epsilon}\right)  \rightarrow\left(
1-\epsilon\right)  ^{1/\gamma_{2}},$ as $n\rightarrow\infty,$ together with
\ $(\ref{approx}),$ $\left(  \ref{limit}\right)  $ and Potter's inequality
(see Lemma \ref{Lem3}), give%
\[
\frac{\sqrt{k}T_{n1}\left(  \epsilon\right)  }{\Lambda\left(  U\left(
n/k\right)  \right)  }=O_{\mathbf{p}}\left(  1\right)  \frac{\overline
{F}\left(  U\left(  n/k\right)  \right)  }{C\left(  U\left(  n/k\right)
\right)  }\frac{\gamma}{\gamma_{1}}\int_{1}^{\infty}z^{2/\gamma_{2}}%
d\frac{F_{n}\left(  zu_{n,\epsilon}\right)  }{\overline{F}\left(  U\left(
n/k\right)  \right)  }.
\]
The expectation of the integral in the previous equation equals%
\[
-\int_{1}^{\infty}z^{2/\gamma_{2}}d\left(  \overline{F}\left(  zu_{n,\epsilon
}\right)  /\overline{F}\left(  U\left(  n/k\right)  \right)  \right)  ,
\]
which, by routine manipulations and the fact that the parameters $\gamma_{1}$
and $\gamma_{2}$ are such that $\gamma_{1}<\gamma_{2},$ converges to $-\left(
1-\epsilon\right)  ^{-1/\gamma}\gamma_{2}/\left(  2\gamma-\gamma_{2}\right)  $
as $n\rightarrow\infty.$ On the other hand, we have $\overline{F}\left(
U\left(  n/k\right)  \right)  =k/n$ and $\left(  k/n\right)  /C\left(
U\left(  n/k\right)  \right)  \rightarrow0$ as $n\rightarrow\infty$ (from
Lemma \ref{lem1}). Therefore, $\sqrt{k}T_{n1}\left(  \epsilon\right)
/\Lambda\left(  U\left(  n/k\right)  \right)  \overset{\mathbf{P}}%
{\rightarrow}0$ as $n\rightarrow\infty$ and so does $\sqrt{k}T_{n1}%
/\Lambda\left(  U\left(  n/k\right)  \right)  .$ Similar arguments lead to the
same result for $\sqrt{k}T_{n2}/\Lambda\left(  U\left(  n/k\right)  \right)
,$ therefore we omit details. Finally, we focus on the third term $S_{n3},$
for which an integration by parts yields%
\begin{align*}
S_{n3}  &  =\int_{X_{n-k:n}}^{U\left(  n/k\right)  }\frac{\overline{F}%
_{n}\left(  z\right)  -\overline{F}\left(  z\right)  }{C^{2}\left(  z\right)
}dC\left(  z\right)  .\\
&  +\frac{\overline{F}_{n}\left(  X_{n-k:n}\right)  -\overline{F}\left(
X_{n-k:n}\right)  }{C\left(  X_{n-k:n}\right)  }-\frac{\overline{F}_{n}\left(
U\left(  n/k\right)  \right)  -\overline{F}\left(  U\left(  n/k\right)
\right)  }{C\left(  U\left(  n/k\right)  \right)  }.
\end{align*}
Changing variables and using the process $\alpha_{n}\left(  \cdot\right)  ,$
we get%
\begin{align*}
\frac{\sqrt{k}S_{n3}}{\Lambda\left(  U\left(  n/k\right)  \right)  }  &
=\frac{k/n}{C\left(  U\left(  n/k\right)  \right)  \Lambda\left(  U\left(
n/k\right)  \right)  }\\
&  \times\left\{  \int_{X_{n-k:n}/U\left(  n/k\right)  }^{1}\frac{\alpha
_{n}\left(  \dfrac{n}{k}\overline{F}\left(  zU\left(  n/k\right)  \right)
\right)  }{\left[  C\left(  zU\left(  n/k\right)  \right)  /C\left(  U\left(
n/k\right)  \right)  \right]  ^{2}}d\left(  \frac{C\left(  zU\left(
n/k\right)  \right)  }{C\left(  U\left(  n/k\right)  \right)  }\right)
\right. \\
&  \ \ \ \ \ \ \ \ \ \ \ \ \ \ \ \ \ \ \ \ \ \left.  +\frac{C\left(  U\left(
n/k\right)  \right)  }{C\left(  X_{n-k:n}\right)  }\alpha_{n}\left(  \dfrac
{n}{k}\overline{F}\left(  X_{n-k:n}\right)  \right)  -\alpha_{n}\left(
1\right)  \right\}  .
\end{align*}
For convenience, we set
\[
\xi_{n}^{+}:=\max(X_{n-k:n}/U\left(  n/k\right)  ,1)\text{ and }\xi_{n}%
^{-}:=\min(X_{n-k:n}/U\left(  n/k\right)  ,1).
\]
By using routine manipulations, including Potter's inequality (see Lemma
\ref{Lem3}) and the fact that $\sup_{0<t<1}\alpha_{n}\left(  t\right)  $ is
stochastically bounded, we show that%
\[
\frac{\sqrt{k}S_{n3}}{\Lambda\left(  U\left(  n/k\right)  \right)  }%
\approx\frac{\gamma}{\gamma_{1}}\left\{  O_{\mathbf{p}}\left(  1\right)
\left(  \xi_{n}^{+}-\xi_{n}^{-}\right)  +\frac{C\left(  U\left(  n/k\right)
\right)  }{C\left(  X_{n-k:n}\right)  }\alpha_{n}\left(  \dfrac{n}{k}%
\overline{F}\left(  X_{n-k:n}\right)  \right)  -\alpha_{n}\left(  1\right)
\right\}  .
\]
Since $\xi_{n}^{+}-\xi_{n}^{-}=\left\vert 1-X_{n-k:n}/U\left(  n/k\right)
\right\vert $ and $X_{n-k:n}/U\left(  n/k\right)  \overset{\mathbf{p}%
}{\rightarrow}1,$ then $\xi_{n}^{+}-\xi_{n}^{-}\overset{\mathbf{p}%
}{\rightarrow}0.$ Now, it is clear that%
\begin{align*}
\frac{\sqrt{k}S_{n3}}{\Lambda\left(  U\left(  n/k\right)  \right)  }  &
=\frac{\gamma}{\gamma_{1}}\left\{  \frac{C\left(  U\left(  n/k\right)
\right)  }{C\left(  X_{n-k:n}\right)  }\left(  \alpha_{n}\left(  \dfrac{n}%
{k}\overline{F}\left(  X_{n-k:n}\right)  \right)  -\alpha_{n}\left(  1\right)
\right)  \right. \\
&  \left.  +\left(  \frac{C\left(  U\left(  n/k\right)  \right)  }{C\left(
X_{n-k:n}\right)  }-1\right)  \alpha_{n}\left(  1\right)  \right\}
+o_{\mathbf{p}}\left(  1\right)  .
\end{align*}
We have $C\left(  U\left(  n/k\right)  \right)  /C\left(  X_{n-k:n}\right)
\overset{\mathbf{p}}{\rightarrow}1$ and $\alpha_{n}\left(  1\right)
=O_{\mathbf{p}}\left(  1\right)  ,$ then it suffices to show that $\alpha
_{n}\left(  \dfrac{n}{k}\overline{F}\left(  X_{n-k:n}\right)  \right)
-\alpha_{n}\left(  1\right)  \overset{\mathbf{p}}{\rightarrow}0.$ Indeed,
making use of the approximation $(\ref{approx}),$ we get
\[
\alpha_{n}\left(  \dfrac{n}{k}\overline{F}\left(  X_{n-k:n}\right)  \right)
-\alpha_{n}\left(  1\right)  =W_{1}\left(  \dfrac{n}{k}\overline{F}\left(
X_{n-k:n}\right)  \right)  -W_{1}\left(  1\right)  +o_{\mathbf{p}}\left(
1\right)  .
\]
Since $\left\{  W_{1}\left(  t\right)  ,\text{ }0\leq t\leq1\right\}  $ is a
Wiener process, then it is easy to verify that%
\[
\left\vert W_{1}\left(  \dfrac{n}{k}\overline{F}\left(  X_{n-k:n}\right)
\right)  -W_{1}\left(  1\right)  \right\vert \overset{d}{=}\left\vert
W_{1}\left(  \left\vert \dfrac{n}{k}\overline{F}\left(  X_{n-k:n}\right)
-1\right\vert \right)  \right\vert .
\]
Recall that $\dfrac{n}{k}\overline{F}\left(  X_{n-k:n}\right)  \overset
{\mathbf{p}}{\rightarrow}1,$ then by using similar arguments as those used in
the proof of Lemma 5.2 (i) in \cite{BMN-2014}, we show that $W_{1}\left(
\left\vert \dfrac{n}{k}\overline{F}\left(  X_{n-k:n}\right)  -1\right\vert
\right)  $ tends to zero in probability, which implies that $\sqrt{k}%
S_{n3}/\Lambda\left(  U\left(  n/k\right)  \right)  \overset{\mathbf{p}%
}{\rightarrow}0$ as well. In summary, we showed that%
\[
\sqrt{k}\mathbf{D}_{n}^{\ast}=\frac{\gamma_{1}\gamma_{2}}{\left(  \gamma
_{1}+\gamma_{2}\right)  ^{2}}\int_{0}^{1}s^{-\gamma/\gamma_{2}-1}W_{1}\left(
s\right)  ds+\frac{\gamma}{\gamma_{1}}W_{1}\left(  1\right)  +o_{p}\left(
1\right)  ,
\]
which leads to the wanted result. Finally, with some elementary calculations,
we get the variance of the Gaussian variable $\sqrt{k}\left(  \overline
{\mathbf{F}}_{n}\left(  X_{n-k:n}\right)  /\overline{\mathbf{F}}\left(
X_{n-k:n}\right)  -1\right)  $ and conclude the proof.$\hfill\square$

\subsection{Proof of Theorem $\ref{Theorem3}$}

For the sake of notational simplicity, we set $\ell=\ell_{n}:=U\left(
n/k\right)  .$ Let us rewrite $\Pi$ into%
\[
\Pi=u\overline{\mathbf{F}}\left(  u\right)  \int_{1}^{\infty}\frac
{\overline{\mathbf{F}}\left(  ux\right)  }{\overline{\mathbf{F}}\left(
u\right)  }dx,
\]
and consider the decomposition%
\[
\frac{\widehat{\Pi}_{n}-\Pi}{\left(  u/\ell\right)  ^{1-1/\gamma_{1}}%
\ell\overline{\mathbf{F}}\left(  \ell\right)  }=\sum_{i=1}^{7}S_{ni},
\]
where%
\[%
\begin{array}
[c]{cl}%
S_{n1} & :=\left\{  \dfrac{\left(  u/X_{n-k:n}\right)  ^{1-1/\widehat{\gamma
}_{1}}}{\left(  u/\ell\right)  ^{1-1/\gamma_{1}}}-1\right\}  \dfrac
{\widehat{\gamma}_{1}}{1-\widehat{\gamma}_{1}}\dfrac{X_{n-k:n}}{\ell}%
\dfrac{\overline{\mathbf{F}}_{n}\left(  X_{n-k:n}\right)  }{\overline
{\mathbf{F}}\left(  \ell\right)  },\medskip\\
S_{n2} & :=\dfrac{X_{n-k:n}}{\ell}\dfrac{\overline{\mathbf{F}}_{n}\left(
X_{n-k:n}\right)  }{\overline{\mathbf{F}}\left(  \ell\right)  }\left\{
\dfrac{\widehat{\gamma}_{1}}{1-\widehat{\gamma}_{1}}-\dfrac{\gamma_{1}%
}{1-\gamma_{1}}\right\}  ,\medskip\\
S_{n3} & :=\dfrac{\gamma_{1}}{1-\gamma_{1}}\dfrac{\overline{\mathbf{F}}\left(
X_{n-k:n}\right)  }{\overline{\mathbf{F}}\left(  \ell\right)  }\dfrac
{\overline{\mathbf{F}}_{n}\left(  X_{n-k:n}\right)  }{\overline{\mathbf{F}%
}\left(  X_{n-k:n}\right)  }\left\{  \dfrac{X_{n-k:n}}{\ell}-1\right\}
,\medskip\\
S_{n4} & :=\dfrac{\gamma_{1}}{1-\gamma_{1}}\dfrac{\overline{\mathbf{F}}%
_{n}\left(  X_{n-k:n}\right)  }{\overline{\mathbf{F}}\left(  X_{n-k:n}\right)
}\left\{  \dfrac{\overline{\mathbf{F}}\left(  X_{n-k:n}\right)  }%
{\overline{\mathbf{F}}\left(  \ell\right)  }-\left(  \dfrac{X_{n-k:n}}{\ell
}\right)  ^{-1/\gamma_{1}}\right\}  \medskip\\
S_{n5} & :=\dfrac{\gamma_{1}}{1-\gamma_{1}}\dfrac{\overline{\mathbf{F}}%
_{n}\left(  X_{n-k:n}\right)  }{\overline{\mathbf{F}}\left(  X_{n-k:n}\right)
}\left\{  \left(  \dfrac{X_{n-k:n}}{\ell}\right)  ^{-1/\gamma_{1}}-1\right\}
,\medskip
\end{array}
\]%
\[%
\begin{tabular}
[c]{ll}%
$S_{n6}$ & $:=\dfrac{\gamma_{1}}{1-\gamma_{1}}\left\{  \dfrac{\overline
{\mathbf{F}}_{n}\left(  X_{n-k:n}\right)  }{\overline{\mathbf{F}}\left(
X_{n-k:n}\right)  }-1\right\}  ,\medskip$\\
$S_{n7}$ & $:=\dfrac{\gamma_{1}}{1-\gamma_{1}}-\left(  \dfrac{u}{\ell}\right)
^{1/\gamma_{1}}\dfrac{\overline{\mathbf{F}}\left(  u\right)  }{\overline
{\mathbf{F}}\left(  \ell\right)  }%
{\displaystyle\int_{1}^{\infty}}
\dfrac{\overline{\mathbf{F}}\left(  ux\right)  }{\overline{\mathbf{F}}\left(
u\right)  }dx.$%
\end{tabular}
\ \ \ \ \ \
\]
We start by representing the five quantities $\sqrt{k}S_{ni},$ $i=1,2,3,5,6$
in terms of the Gaussian processes $W_{1}$ and $W_{2},$ given in Theorem
\ref{Theorem1}, then we show that $\sqrt{k}S_{n4}$ and $\sqrt{k}S_{n7}$
converge to deterministic limits. For the first term $S_{n1},$ recall that
$X_{n-k:n}\approx\ell,$ which implies by the regular variation of
$\overline{\mathbf{F}}$ that $\overline{\mathbf{F}}\left(  X_{n-k:n}\right)
\approx\overline{\mathbf{F}}\left(  \ell\right)  .$ On the other hand, we have
$\widehat{\gamma}_{1}\overset{\mathbf{P}}{\rightarrow}\gamma_{1}$ and, from
Remark \ref{Rmrk}, $\overline{\mathbf{F}}_{n}\left(  X_{n-k:n}\right)
\approx\overline{\mathbf{F}}\left(  X_{n-k:n}\right)  .$ It follows that%
\[
S_{n1}\approx\frac{\gamma_{1}}{1-\gamma_{1}}\left\{  \frac{\left(
u/X_{n-k:n}\right)  ^{1-1/\widehat{\gamma}_{1}}}{\left(  u/\ell\right)
^{1-1/\gamma_{1}}}-1\right\}  =\frac{\gamma_{1}}{1-\gamma_{1}}\left\{
S_{n1}^{\left(  1\right)  }+S_{n1}^{\left(  2\right)  }\right\}  ,
\]
where $S_{n1}^{\left(  1\right)  }:=\left(  u/\ell\right)  ^{1/\gamma
_{1}-1/\widehat{\gamma}_{1}}-1$ and $S_{n1}^{\left(  2\right)  }:=\left(
u/\ell\right)  ^{1/\gamma_{1}-1/\widehat{\gamma}_{1}}\left(  \left(
\ell/X_{n-k:n}\right)  ^{1-1/\widehat{\gamma}_{1}}-1\right)  .$ By using the
mean value theorem in $S_{n1}^{\left(  1\right)  },$ we have%
\[
S_{n1}^{\left(  1\right)  }=\left(  1/\gamma_{1}-1/\widehat{\gamma}%
_{1}\right)  \left(  u/\ell\right)  ^{\epsilon_{n}}\log\left(  u/\ell\right)
,
\]
with $\epsilon_{n}$ being between $1/\gamma_{1}-1/\widehat{\gamma}_{1}$ and
$0.$ The consistency of $\widehat{\gamma}_{1}$ implies that $\epsilon
_{n}\overset{\mathbf{P}}{\rightarrow}0,$ and therefore $S_{n1}^{\left(
1\right)  }\approx\gamma_{1}^{-2}\left(  \widehat{\gamma}_{1}-\gamma
_{1}\right)  \log\left(  u/\ell\right)  .$ Likewise, we may readily show that
\[
S_{n1}^{\left(  2\right)  }\approx\frac{1-\gamma_{1}}{\gamma_{1}}\left(
\frac{X_{n-k:n}}{\ell}-1\right)  .
\]
\ Consequently,%
\[
S_{n1}\approx\frac{\log\left(  u/\ell\right)  }{\gamma_{1}\left(  1-\gamma
_{1}\right)  }\left(  \widehat{\gamma}_{1}-\gamma_{1}\right)  +\left(
\frac{X_{n-k:n}}{\ell}-1\right)  .
\]
By using similar arguments we also show that
\[
S_{n2}\approx\frac{\widehat{\gamma}_{1}-\gamma_{1}}{\left(  1-\gamma
_{1}\right)  ^{2}},\text{ }S_{n3}\approx\frac{\gamma_{1}}{1-\gamma_{1}}\left(
\frac{X_{n-k:n}}{\ell}-1\right)  \text{ and }S_{n5}\approx-\frac{1}%
{1-\gamma_{1}}\left(  \frac{X_{n-k:n}}{\ell}-1\right)  .
\]
Summing these four terms, we obtain%
\[
S_{n1}+S_{n2}+S_{n3}+S_{n5}\approx\frac{\left(  1-\gamma_{1}\right)
\log\left(  u/\ell\right)  +\gamma_{1}}{\gamma_{1}\left(  1-\gamma_{1}\right)
^{2}}\left(  \widehat{\gamma}_{1}-\gamma_{1}\right)  .
\]
Now, we use the second approximation in Theorem \ref{Theorem1} to have%
\begin{equation}%
\begin{array}
[c]{l}%
\sqrt{k}\left(  S_{n1}+S_{n2}+S_{n3}+S_{n5}\right)  \approx\dfrac{\left(
1-\gamma_{1}\right)  \log a+\gamma_{1}}{\gamma_{1}\left(  1-\gamma_{1}\right)
^{2}}\\
\multicolumn{1}{c}{\times\left\{  \int_{0}^{1}t^{-1}\left(  cW_{1}\left(
t\right)  -c_{2}W_{2}\left(  t\right)  \right)  dt-cW_{1}\left(  1\right)
+c_{2}W_{2}\left(  1\right)  +\mu\left(  k\right)  +o_{\mathbf{p}}\left(
1\right)  \right\}  .}%
\end{array}
\label{S1235}%
\end{equation}
The asymptotic representation of Theorem \ref{Theorem2} yields%
\begin{equation}
\sqrt{k}S_{n6}\approx\frac{\gamma_{1}}{1-\gamma_{1}}\left\{  \frac{\gamma
_{1}\gamma_{2}}{\left(  \gamma_{1}+\gamma_{2}\right)  ^{2}}\int_{0}%
^{1}s^{-\gamma/\gamma_{2}-1}W_{1}\left(  s\right)  ds+\frac{\gamma}{\gamma
_{1}}W_{1}\left(  1\right)  \right\}  +o_{\mathbf{p}}\left(  1\right)  .
\label{S6}%
\end{equation}
For the fourth term $S_{n4},$ it suffices to use the second-order condition of
regular variation $\left(  \ref{secon-orderFbold}\right)  $ and the fact that
$X_{n-k:n}\approx\ell,$ to get%
\begin{equation}
\sqrt{k}S_{n4}=o_{\mathbf{P}}\left(  \sqrt{k}\mathbf{A}\left(  \ell\right)
\right)  =o_{\mathbb{P}}\left(  1\right)  ,\text{ as }n\rightarrow\infty.
\label{S4}%
\end{equation}
For the last term $S_{n7},$ we first note that%
\[
S_{n7}=\int_{1}^{\infty}x^{-1/\gamma_{1}}dx-\frac{u\overline{\mathbf{F}%
}\left(  u\right)  }{\left(  u/\ell\right)  ^{1-1/\gamma_{1}}\ell
\overline{\mathbf{F}}\left(  \ell\right)  }\int_{1}^{\infty}\frac
{\overline{\mathbf{F}}\left(  ux\right)  }{\overline{\mathbf{F}}\left(
u\right)  }dx,
\]
In addition to the the regular variation of $\left\vert \mathbf{A}\right\vert
,$ we apply the uniform inequality of regularly varying functions
\citep[see, e.g., Theorem 2.3.9 in][page 48]{deHF06} to show that%
\begin{equation}
\sqrt{k}S_{n7}\sim\frac{\sqrt{k}\mathbf{A}\left(  \ell\right)  }{\left(
\gamma_{1}-1-\tau_{1}\right)  \left(  \gamma_{1}-1\right)  }. \label{S7}%
\end{equation}
Finally, gathering results $\left(  \ref{S1235}\right)  ,$ $\left(
\ref{S6}\right)  ,$ $\left(  \ref{S4}\right)  $ and $\left(  \ref{S7}\right)
$ yields a Gaussian approximation from which we derive the normal limiting
distribution of the premium estimator $\widehat{\Pi}_{n}.$ Tedious
computations for the asymptotic variance complete the proof of the
theorem.\hfill$\Box$\textbf{\medskip}

\noindent\textbf{Concluding notes\medskip}

\noindent We proposed an estimator of the tail index for randomly truncated
heavy-tailed data based on the same number of extreme observations from both
truncated and truncation variables. Thus, the determination of the optimal
sample fraction becomes standard, in the sense of applying any convenient
algorithm available in the literature. The asymptotic normality of the
estimator is established by taking into account the dependence structure of
the observations and a practical way to construct confidence bounds for the
extreme value index is given.\ The obtained Gaussian approximations are of
great usefulness as they allow to determine the limiting distributions of
several statistics related to the extreme value index such that high quantiles
and risk measures estimators (see, for instance, \citeauthor{NM-2009},
\citeyear{NM-2009}). As an application, we provided an estimator for the
excess-of-loss reinsurance premium in the case of large randomly truncated claims.

\section{Appendix}

\begin{lemma}
\label{lem1}Assume that the second-order conditions $(\ref{second-order})$
hold with $\gamma_{1}<\gamma_{2}.$ Then the function $C$ is regularly varying
at infinity with index $-1/\gamma_{2}$ and $t^{-1}C\left(  U\left(  t\right)
\right)  \rightarrow0$ as $t\rightarrow\infty.$
\end{lemma}

\begin{proof}
We have $C=\overline{G}-\overline{F}$ with $\gamma_{1}<\gamma_{2},$ hence
$C\left(  x\right)  \sim\overline{G}\left(  x\right)  $ as $x\rightarrow
\infty.$ Since both $\overline{F}$ and $\overline{G}$ satisfy the second-order
conditions $(\ref{second-order}),$ then in view of Lemma 3 in \cite{HJ-2011},
there exist two constants $\delta,\delta_{2}>0,$ such that $\overline
{F}\left(  x\right)  \sim\delta x^{-1/\gamma}$ and $\overline{G}\left(
x\right)  \sim\delta_{2}x^{-1/\gamma_{2}},$ as $x\rightarrow\infty.$ The first
equivalence implies that $U\left(  t\right)  \sim\delta^{\gamma}t^{\gamma},$
as $t\rightarrow\infty,$ therefore $C\left(  U\left(  t\right)  \right)
\sim\delta_{2}\delta^{\gamma/\gamma_{2}}t^{\gamma/\gamma_{2}},$ it follows
that%
\[
t^{-1}C\left(  U\left(  t\right)  \right)  \sim\delta_{2}\delta^{\gamma
/\gamma_{2}}t^{\gamma/\gamma_{2}-1}.
\]
By assumption, we have $\gamma_{1}<\gamma_{2},$ it follows that $\gamma
/\gamma_{2}=\gamma_{1}/\left(  \gamma_{1}+\gamma_{2}\right)  $ is less to
$1/2,$ thus $t^{-1}C\left(  U\left(  t\right)  \right)  \rightarrow0$ as
$t\rightarrow\infty,$ which achieves the proof of the lemma.
\end{proof}

\begin{lemma}
\label{lem2}Under the assumptions of Lemma $\ref{lem1},$ we have%
\begin{equation}
t\Lambda\left(  U\left(  t\right)  \right)  C\left(  U\left(  t\right)
\right)  \rightarrow\gamma_{1}/\gamma\text{ as }t\rightarrow\infty.
\label{limit}%
\end{equation}

\end{lemma}

\begin{proof}
Recalling that $\Lambda\left(  x\right)  =$ $\int_{x}^{\infty}dF\left(
z\right)  /C\left(  z\right)  $ and $\overline{F}\left(  U\left(  t\right)
\right)  =t^{-1},$ we write%
\[
t\Lambda\left(  U\left(  t\right)  \right)  C\left(  U\left(  t\right)
\right)  =-\int_{1}^{\infty}\frac{C\left(  U\left(  t\right)  \right)
}{C\left(  zU\left(  t\right)  \right)  }\frac{d\overline{F}\left(  zU\left(
t\right)  \right)  }{\overline{F}\left(  U\left(  t\right)  \right)  }.
\]

Making use of Potter's inequality\ for both $C$ and $\overline{F},$ we infer
that,
\[
t\Lambda\left(  U\left(  t\right)  \right)  C\left(  U\left(  t\right)
\right)  \sim-\int_{1}^{\infty}z^{1/\gamma_{2}}dz^{-1/\gamma}=\gamma
_{1}/\gamma,\text{ as }t\rightarrow\infty,
\]
as sought.
\end{proof}

\begin{lemma}
\textbf{\label{Lem3}}Suppose that $\varphi$ is a regularly varying function
(at infinity) with index $\rho\in\mathbb{R},$ i.e. $\varphi\left(  tx\right)
/\varphi\left(  t\right)  \rightarrow x^{\rho},$ as $t\rightarrow\infty,$ for
all $x>0.$ Then for any $0<\epsilon<1,$ there exists $t_{0}=t_{0}\left(
\epsilon\right)  $ such that for $t\geq t_{0},$ $tx\geq t_{0},$
\[
\left(  1-\epsilon\right)  x^{\varrho}\min\left(  x^{\epsilon},x^{-\epsilon
}\right)  <\frac{\varphi\left(  tx\right)  }{\varphi\left(  t\right)
}<\left(  1+\epsilon\right)  x^{\rho}\max\left(  x^{\epsilon},x^{-\epsilon
}\right)  .
\]
In other words, we have, for every $x_{0}>0,$%
\[
\lim_{t\rightarrow\infty}\sup_{x\geq x_{0}}\left\vert \frac{\varphi\left(
tx\right)  }{\varphi\left(  t\right)  }-x^{\rho}\right\vert =0.
\]

\end{lemma}

\begin{proof}
This result, known as Potter's bound inequalities, is stated in, for instance,
\citeauthor{deHF06}, \citeyear{deHF06}, Proposition B.1.9, Assertion 5, page 367.
\end{proof}

\end{document}